\documentclass{compositio}

\usepackage{amsmath,amsfonts,yhmath,hyperref}
\usepackage{amsthm}
\usepackage{amssymb}
\usepackage[latin9]{inputenc}
\usepackage[nottoc,notlof,notlot]{tocbibind}
\usepackage{mathtools}
\usepackage{enumitem}
\usepackage{tikz-cd}
\usepackage{multirow}
\usepackage{dsfont} 

\usepackage{pgf,tikz}

\usetikzlibrary{arrows}
\usetikzlibrary[patterns]

\definecolor{qqqqff}{rgb}{0.,0.,1.}
\definecolor{cqcqcq}{rgb}{0.7529411764705882,0.7529411764705882,0.7529411764705882}
\definecolor{ttqqqq}{rgb}{0.2,0.,0.}
\definecolor{qqqqff}{rgb}{0.,0.,1.}
\definecolor{xdxdff}{rgb}{0.49019607843137253,0.49019607843137253,1.}
\definecolor{zzttqq}{rgb}{0.6,0.2,0.}
\definecolor{cqcqcq}{rgb}{0.7529411764705882,0.7529411764705882,0.7529411764705882}

\definecolor{yqyqyq}{rgb}{0.5019607843137255,0.5019607843137255,0.5019607843137255}
\definecolor{uuuuuu}{rgb}{0.26666666666666666,0.26666666666666666,0.26666666666666666}
\definecolor{xdxdff}{rgb}{0.49019607843137253,0.49019607843137253,1.}
\definecolor{qqqqff}{rgb}{0.,0.,1.}

\newcommand{\ZZ}{\mathbb{Z}}
\newcommand{\RR}{\mathbb{R}}
\newcommand{\CC}{\mathbb{C}}

\newcommand{\MM}{\mathcal{M}}
\newcommand{\oo}{\mathfrak{o}}

\newcommand{\ev}{\mathrm{ev}}
\newcommand{\mom}{\mathrm{mom}}
\newcommand{\rkn}{\mathrm{rk}N }

\newcommand{\gen}[1]{\langle #1 \rangle}

\newcommand{\re}{\mathfrak{Re}}
\newcommand{\im}{\mathfrak{Im}}

\newcommand{\bino}[2]{\begin{pmatrix}
#1 \\
#2 \\
\end{pmatrix}}

\newtheorem{theo}{Theorem}[section]
\newtheorem*{theom}{Theorem}
\newtheorem{prop}[theo]{Proposition}
\newtheorem{coro}[theo]{Corollary}
\newtheorem{lem}[theo]{Lemma}

\theoremstyle{definition}
\newtheorem{defi}[theo]{Definition}

\theoremstyle{remark}
\newtheorem{remark}[theo]{Remark}
\newenvironment{rem}[1]{
    \begin{remark}#1}{
    \xqed{\blacklozenge}\end{remark}
}

\theoremstyle{remark}
\newtheorem{example}[theo]{Example}
\newenvironment{expl}[1]{
    \begin{example}#1}{
    \xqed{\lozenge}\end{example}
}

\newcommand{\xqed}[1]{
    \leavevmode\unskip\penalty9999 \hbox{}\nobreak\hfill
    \quad\hbox{\ensuremath{#1}}}

\keywords{Enumerative geometry, refined invariants\\ \textit{Data Statement:} I do not have any data to point.}
 
\begin{document}
 
%\renewcommand{\qedsymbol}{\filledbox}
%Good resources for looking up how to do stuff:
%Binary operators: http://www.access2science.com/latex/Binary.html
%General help: http://en.wikibooks.org/wiki/LaTeX/Mathematics
%Or just google stuff
 
\title{Refined count of oriented real rational curves}
\author{Thomas Blomme}

\begin{abstract}
We introduce a \textit{quantum index} for oriented real curves inside toric varieties. This quantum index is related to the computation of the area of the amoeba of the curve for some chosen $2$-form. We then make a refined signed count of oriented real rational curves solution to some enumerative problem. This generalizes the results from \cite{mikhalkin2017quantum} to higher dimension. Finally, we use the tropical approach to relate these new refined invariants to previously known tropical refined invariants.
\end{abstract}

\maketitle

\tableofcontents

\section{Introduction}

\subsection{Setting and enumeration of real rational curves}

Let $N$ be some lattice and $\Delta=(n_i)\subset N$ a family of primitive vectors of total sum $0$. We consider oriented real rational curves of degree $2\Delta=(2n_i)$ in some toric variety $\CC\Delta$ built from $\Delta$ and with a dense torus orbit $N_{\CC^*}=N\otimes\CC^*$. This toric variety is endowed with the complex conjugation, making it into a real variety whose real locus is denoted by $\RR\Delta$. As the degree is $2\Delta$, the curves we consider are tangent to the toric divisors at any of the intersection points they have with them. Moreover, the real locus of each such curve lies in a unique orthant of the real locus $\RR\Delta$, and we assume it to be the positive orthant $N_{\RR_+^*}=N\otimes\RR_+^*\subset\RR\Delta$. Finally, we assume that the curves have $R+1$ real intersection points with the boundary, and $S$ pairs of complex conjugated intersection points, so that the dimension of the space of such curves is
$$\rkn+|\Delta|-3=\rkn-2+R+2S.$$

\medskip

In the particular setting of the planar case $\rkn=2$, in \cite{mikhalkin2017quantum}, G. Mikhalkin considers the following enumerative problem: fix $|\Delta|-1$ points on the toric boundary and look for rational curves passing through them. More precisely, choose $R+1$ points on the boundary of the positive orthant $N_{\RR_+^*}$ and $S$ pairs of complex conjugated points on the toric boundary, such that
\begin{itemize}[label=-]
\item each toric divisor has a number of points corresponding to the intersection with a curve of degree $\Delta$,
\item the chosen points satisfy the \textit{Menelaus condition}, so that there exists at least one such curve. This condition is merely that the product of the \textit{coordinates} of all the points is $1$. Here, the \textit{coordinate} denote the evaluation of some specific monomial $m\in M=N^*$.
\end{itemize}
The second condition comes from the fact that for any curve of degree $\Delta$, the position of the last intersection point is fully determined by the others. Concretely, it means that the choice of the position of the first intersection point is not part of the constraints but given by the problem but given by the choice of the other constraints. Then, which are the real rational curves of degree $2\Delta$ passing through the configuration ?

\begin{expl}
Consider the toric surface $\CC P^2$ and $\Delta=\Delta_d=\{(-e_1)^d,(-e_2)^d,(e_1+e_2)^d\}$, the degree associated to degree $d$ curves. Then, we choose $3d$ points in $\CC P^2$, real or pairs of conjugated ones, with exactly $d$ points on each coordinate axis: $[a_i:0:1]$, $[0:1:b_i]$ and $[1:c_i:0]$ for $1\leqslant i\leqslant d$ with $a_i,b_i,c_i>0$. We assume the Menelaus condition: $\prod_1^d a_ib_ic_i=1$. We are thus looking for curves of degree $2d$ tangent to each coordinate axis and passing through the point configuration.
\end{expl}

This enumerative problem can be seen as a degenerated case of the usual enumerative problem consisting in counting rational curves passing through the right number of points chosen in generic position inside a toric surface, and then taking the square map $\mathrm{Sq}:\CC\Delta\rightarrow\CC\Delta$ that squares all coordinates, to get curves which are tangent to the toric divisors where they meet.%In the case of $\CC P^2$, it means finding rational curves of degree $d$ passing through $3d-1$ points. If the points are chosen complex, the number of complex curves, which does not depend on the chosen points, is $N_d$, a Gromov-Witten invariant of $\CC P^2$. The result translates to any toric surface. If the point configuration is chosen real, meaning that it consists of real points and pairs of complex conjugated points, and we look for real curves, the number of curves passing through the configuration might depend on the choice of the points. However, J-Y.~Welschinger showed that if $\CC\Delta$ is a del Pezzo surface, we can still get a count that does not depend on the position of the points, but only on the number of pairs of complex conjugated points $s$, provided that we count curves with a suitable sign. The corresponding invariant is denoted by $W_{d,s}$.

%\medskip

%Using the correspondence theorem \cite{mikhalkin2005enumerative} and the tropical approach, G.~Mikhalkin managed to provide an algorithm to compute $N_d$ and $W_{d,0}$. This algorithm consists in counting tropical curves solution to a similar enumerative problem with a suitable choice of multiplicity. Later, F.~Block and L.~G\"ottsche \cite{block2016refined} found a way to merge both choices of multiplicities in a common Laurent polynomial multiplicity, which surprisingly also provides an invariant on the tropical side, as proven by I.~Itenberg and G.~Mikhalkin in \cite{itenberg2013block}. Concretely, the Mikhalkin multiplicity is a product of vertex multiplicities $m^\CC_W$, and the Block-G\"ottsche multiplicity is obtained by replacing each vertex multiplicity $m^\CC_W$ by its quantum analog $\frac{q^{m^\CC_W/2}-q^{-m^\CC_W/2}}{q^{1/2}-q^{-1/2}}$. We only know a glimpse of the role of these tropical invariants in classical geometry, and only in a few specific situations.

\medskip

In \cite{mikhalkin2017quantum}, Mikhalkin introduced a quantum index (a discrete number associated to each real type I curve) and a sign, such that the signed count of the solutions to the considered enumerative problem, sorted out by the value of their quantum index, does not depend on the choice of the points on the boundary of the toric surface, as long as this choice is generic. Similarly to the case of Welschinger invariants \cite{welschinger2005invariants}, the count does only depend on the number of pairs of complex conjugated points on each toric divisor. In case all the chosen points are real, Mikhalkin related through the use of his correspondence theorem \cite{mikhalkin2005enumerative} the refined signed count to a tropical count using the refined Block-G\"ottsche multiplicity \cite{block2016refined}. In presence of pairs of complex points, the relation to tropical invariants was proved by the author in  \cite{blomme2020tropical} and \cite{blomme2020computation}.

\medskip

In higher dimension, there is a broader freedom in the choice of the constraints imposed on the curves: instead of imposing the curve to pass through some points, it is possible to ask them to meet some torus suborbits inside the toric variety $\CC\Delta$. Although the complex considerations generalize quite easily to a higher dimensional setting, in the sense that we still have complex invariants and a tropical way to compute them, the search for real invariants and refined tropical invariants is much more difficult. In the tropical setting, this is partly due to the fact that the multiplicity provided by higher versions of the correspondence theorem due to T. Nishinou and B. Siebert \cite{nishinou2006toric}, or I. Tyomkin \cite{tyomkin2017enumeration}, is not given by a simple recipe allowing us to apply the trick yielding the refined Block-G\"ottsche multiplicity. Concretely, it means that the multiplicity is not a product of vertex multiplicities anymore. However, for a specific choice of enumerative problem involving some choice of $2$-form $\omega$, we recover such a simple recipe providing real tropical invariants \cite{mandel2015scattering}. This suggests the existence of a corresponding real refined count which is precisely the content of this paper: we enlarge the definition of the quantum index in higher dimension, and we provide a corresponding real refined count.

\begin{rem}
The existence of tropical refined invariants has been generalized by the author in \cite{blomme2020refined} for enumerative problems which do not rely on the choice of a $2$-form, and for which the complex multiplicity is not a product over the vertices of the tropical curve. The associated real setting might also be conducive to the existence of real refined invariants.
\end{rem}

The real invariants introduced in this paper provide through the use of the correspondence theorem a real counterpart to some tropical invariants which are used in mirror symmetry for cluster varieties and Calabi-Yau manifolds. Briefly, already in the planar setting, in the work of M.~Gross, R.~Pandharipande, B.~Siebert and S.~Keel, the associated complex invariants introduced in \cite{gross2010tropical} are used in \cite{gross2015mirror} to construct mirrors to certain Calabi-Yau manifolds through the use of \textit{scattering diagrams}. Meanwhile, the tropical refined invariants appear in the corresponding scattering for some quantized (or deformed) noncommutative version of the varieties, related to quantum cluster algebras. The quantum versions of \cite{gross2010tropical} and \cite{gross2015mirror} have been done by P.~Bousseau \cite{bousseau2020vertex} \cite{bousseau2020mirror} although quantum scattering diagrams had already been considered by S-A.~Filippini and J.~Stoppa \cite{filippini2012block}. See also the work of T.~Mandel \cite{mandel2015scattering} and \cite{mandel2019tropical} for more details on the use of refined tropical invariants, and higher dimensional scattering diagrams.

\subsection{Quantum indices of real curves}

In the planar setting, the \textit{quantum index} introduced by Mikhalkin in \cite{mikhalkin2017quantum} corresponds to a shift of the equal areas of the amoeba and coamoeba of a real oriented curve. More precisely, let $f:\CC C\dashrightarrow N_{\CC^*}$ be a type I real curve, meaning that $\RR C$ disconnects the Riemann surface $\CC C$ into two connected components $S$ and $\overline{S}$. The choice of one of the connected components $S$ induces an orientation of $\RR C$ as its boundary, called a \textit{complex orientation}. Moreover, we have the logarithm and argument map which map a point of $N_{\CC^*}$ to its logarithm and argument taken coordinate by coordinate:
$$\left\{ \begin{array}{l}
\log:\CC\Delta\dashrightarrow N_\RR , \\
\arg:\CC\Delta\dashrightarrow N\otimes\RR/2\pi\ZZ ,\\
\end{array} \right. $$
The image of $\CC C$ or $S$ by $\log$ (resp. $\arg$) is called the \textit{amoeba} (resp. \textit{coamoeba}). When $N$ is of rank $2$, we have up to sign a canonical choice of volume form $\varpi\in\Lambda^2 M$ which allows us to compute the area of the amoeba:
$$\mathcal{A}_\mathrm{log}(S,\varpi)=\int_{S}(\log\circ f)^*\varpi=\int_{S}(\arg\circ f)^*\varpi,$$
and similarly the area of the coamoeba. Both are in fact equal. Up to a shift by the argument of the coordinates of the complex intersection points with the toric boundary, the log-area happens to be a half-integer multiple of $\pi^2$, called the \textit{quantum index} (see \cite{mikhalkin2017quantum}). In particular, if all intersection points are real, we get a half-integer multiple of $\pi^2$ without having to shift the log-area.

\medskip

In higher dimension, we define with the same formula the log-area for some choice of $2$-form $\varpi$. We thus have several possible generalizations for the quantum index, according to whether the intersection points with the toric boundary are all real or there are some pairs of complex intersection points. If there are complex intersection points, we still compute the log-area $\mathcal{A}_\mathrm{log}(S,\omega)$ with respect to the $2$-form $\omega$ and shift it by some quantity depending on the argument of the complex intersection points and $\omega$. See section \ref{section quantum index} for more details. The $2$-form is here denoted by $\omega$ because in order to get the invariance statement, we need to compute the log-area with respect to the $2$-form used to define the enumerative problem.

\medskip

If all the intersection points are real, it is possible to further refine the quantum index by defining some class $h(S)\in\Lambda^2 N$ that enables the computation of the log-area for any choice of $2$-form $\varpi\in\Lambda^2 M$.

\begin{theom}[\ref{theorem quantum index}]
Let $f:S\dashrightarrow N_{\CC^*}$ be a real oriented curve of degree $\Delta$.
\begin{itemize}[label=$\circ$]
\item If $S$ has only real intersection with the toric boundary, there exists a unique class $h(S)\in\Lambda^2 N$, called the quantum index, such that for any $2$-form $\varpi$ one has
$$\mathcal{A}_{\log}(S,\varpi)=\frac{\pi^2}{2}\langle h(S),\varpi\rangle.$$
\item If $S$ has complex intersection with the toric boundary, let $\varepsilon_j\theta_j\pi$ be the argument of the monomial $\iota_{n_j}\omega$ evaluated at $p_j$ for each complex puncture $p_j$, with $\varepsilon_j=\pm 1$ and $0<\theta_j<1$. Then there exists some half-integer $k(S,\omega)\in\frac{1}{2}\ZZ$ such that
$$\mathcal{A}_{\log}(S,\omega)-\pi^2\sum_j \varepsilon_j(2\theta_j-1)=\pi^2k(S,\omega).$$
\end{itemize}
\end{theom}

Similarly to the planar case, it is possible to compute the quantum index when the curves are of toric type I, which are some specific curves having only real intersection points with the toric boundary, and for rational curves. Moreover, due to its description as an area, the quantum index is naturally additive, and this allows us to compute the quantum indices of curves in a family close to their tropical limit. Hence, the quantum index fits easily in the setting of a correspondence theorem.

\subsection{Refined enumerative geometry and invariance result}

We now precise the setting of the enumerative problem for which we provide an invariance statement. We consider oriented real rational curves of degree $2\Delta$ with real locus inside the positive orthant $N_{\RR_+^*}$ of $\CC\Delta$. Let $\omega$ be a $2$-form chosen generically in the sense of section \ref{section statement of results}, or remark \ref{remark generic choice}. The parameters of the points of intersection between a rational curve and the toric boundary are denoted by $p_0,\dots,p_R,p_1^\pm,\dots,p_S^\pm\in\CC P^1$, where $p_j^\pm$ are pairs of complex conjugated points. We have the evaluation map that maps a curve $f:\CC P^1\dashrightarrow N_{\CC^*}$ to
\begin{itemize}[label=-]
\item a family of $\rkn-2$ monomials evaluated at $f(p_0)$,
\item the monomials $\iota_{n_i^{(+)}}\omega$ evaluated at $f(p_i^{(+)})$, where $n_i^{(+)}$ is the vector in the fan of $\CC\Delta$ associated to the toric divisor to which $f(p_i^{(+)})$ is mapped.
\end{itemize}
We then have
$$\ev:N_\RR\times\MM_{0,R,S}^\tau\longrightarrow(\RR_+^*)^{\rkn-2}\times(\RR_+^*)^R\times(\CC^*)^S=X,$$
where $\MM_{0,R,S}^\tau$ denotes the moduli space of oriented real rational curves with $R+1$ real marked points and $S$ pairs of complex conjugated marked points. Its product with $N_\RR$ parametrizes the oriented real rational curves mapped to $\CC\Delta$ in which we are interested. We can also assume the target space $X$ to be $\RR^{\rkn-2+R}\times(\CC/2i\pi\ZZ)^S$ after composing by the logarithm for each coordinate. The enumerative problem amounts to find the preimage of a generic point $(P,\mu)\in X$. Concretely, it means finding curves such that each intersection point with the boundary belongs to a codimension $2$ suborbit of slope prescribed by $\omega$ and lying inside the toric boundary. The chosen suborbits satisfy some Menelaus condition similar to the planar case, meaning that the suborbit to which $f(p_0)$ belongs is fixed by the other. The other constraint is that $p_0$ is mapped to some fixed point inside this suborbit.

\medskip

We now define a sign associated to each oriented rational curve as follows. For more details, see section \ref{section statement of results}. Choose an orientation $\oo(X)$ of $X$ and an orientation $\oo(N)$ of the lattice $N$. The connected components of $\MM_{0,R,S}^\tau$ are in bijection with the permutations $\gamma\in\mathcal{S}_R$ and subsets $J\subset[\![1;S]\!]$. The permutation $\gamma$ corresponds to the cyclic order induced by the orientation on the real marked points:
$$p_0<p_{\gamma(1)}<\cdots<p_{\gamma(R)}<p_0,$$
and $J$ corresponds to the indices $j$ such that $p_j^-$ belongs to $S$, the preferred component inducing the complex orientation of the curve. The corresponding component is denoted by $\MM(\gamma,J)$. It is possible to define a canonical orientation $\oo^{\gamma,J}(\MM)$ of $\MM(\gamma,J)$, as done in section \ref{section orientation moduli space}. The sign of a curve is then defined as follows, and depends on whether or not the evaluation map preserves the orientations: with $\varepsilon(\gamma)$ being the signature of $\gamma$,
$$\sigma(S)=\varepsilon(\gamma)(-1)^{|J|}\left.\frac{\ev^*\oo(X)}{\oo(N)\oo^{\gamma,J}(\MM)}\right|_S.$$

Then, for a generic generic choice of $(P,\mu)$, we can make the refined signed count of the solutions:
\begin{itemize}[label=$\circ$]
\item If $S=0$, so that there are no complex constraints, we set
$$R_{\Delta,\omega}(P,\mu)=\sum_{S\in\ev^{-1}(P,\mu)}\sigma(S)q^{h(S)}\in \ZZ[\Lambda^2 N],$$
where $h(S)\in\Lambda^2 N$ is the quantum index associated to the oriented real rational curve $S$. This is a Laurent polynomial in $\bino{\rkn}{2}$ variables. To obtain a Laurent polynomial in one variable, one can evaluate a $2$-form $\varpi$ on the exponents. The chosen $2$-form $\varpi$ might be different from $\omega$.
\item If $S\neq 0$, \textit{i.e.} with complex constraints, we set
$$R_{\Delta,\omega}(P,\mu)=\sum_{S\in\ev^{-1}(P,\mu)}\sigma(S)q^{k(S,\omega)}\in \ZZ[q^{\pm 1/2}],$$
where $k(S,\omega)\in\frac{1}{2}\ZZ$ is the quantum index associated to the oriented real rational curve $S$, \textit{i.e.} a shift of its log-area with respect to the form $\omega$ used to define the enumerative problem.
\end{itemize}

We then have the following invariance statement. Notice that it contains invariance for both the totally real case, and in presence of complex constraints.

\begin{theom}[\ref{theorem invariance}]
The value of $R_{\Delta,\omega}(P,\mu)$ does not depend on the choice of $(P,\mu)$ as long as it is generic. It is an invariant denoted by $R_{\Delta,\omega}$.
\end{theom}

\begin{rem}
Notice that implicitly, the data of $\Delta$ includes the first vector directing the first end $n_0$ associated to $p_0$ which has more than an $\omega$-constraint, and the pairs of vectors associated to points $p_j^\pm$ exchanged by the complex conjugation. In particular, the value of the Laurent polynomial depends on the choice of $p_0$ as first point, and the repartition of real and complex constraints among the toric divisors. The latter is not surprising since it is already the case in the planar setting.
\end{rem}

As the polynomial $R_{\Delta,\omega}$ does not depend on the choice of $(P,\mu)$, it is possible to try to compute the invariant by choosing the constraints close to the tropical limit. Applying the correspondence theorem from \cite{blomme2020computation} and doing a readily identical computation, we can then relate the classical invariants $R_{\Delta,\omega}$ to tropical invariants described in \cite{blomme2020refined}. Briefly, they are defined as follows. Let $\Delta_\mathrm{trop}$ be the degree obtained from $\Delta$ by merging the vectors associated to a complex marking: the pair $(n_j^+,n_j^-)$ is replaced by a unique vector $2n_j^+$. We then count the rational tropical curves of degree $\Delta_\mathrm{trop}$ such that the end directed by $n_i$ belongs to a fixed affine hyperplane of slope $\iota_{n_i}\omega$, and the first end directed by $n_0$ belongs to a fixed plane, all constraints being chosen generically. The tropical curves $\Gamma$ are counted with a refined multiplicity
$$m_\Gamma^q=\prod_w(q^{a_w\wedge b_w}-q^{-a_w\wedge b_w}) \in \ZZ[\Lambda^2 N],$$
where the product is indexed by the vertices of $\Gamma$, and $a_w,b_w$ are the slopes of the outgoing edges, in an order such that $\omega(a_w,b_w)>0$

\begin{theom}\cite{blomme2020refined}
The count of solutions using $m_\Gamma^q$ does not depend on the constraints as long as they are chosen generically. The corresponding tropical invariant is denoted by $N_{\Delta_\mathrm{trop},\omega}^{\partial,\mathrm{trop}}$.
\end{theom}

\begin{rem}
\label{remark generic choice}
We can now state the genericity assumption on $\omega$: for any $\Gamma$, if $w$ is not a flat vertex, $\omega(a_w,b_w)\neq 0$.
\end{rem}

Both refined invariants are then related in the following theorem.

\begin{theom}[\ref{theorem tropical computation}]
One has the following:
\begin{itemize}[label=$\circ$]
\item If there are no complex constraints, \textit{i.e.} $S=0$, then
$$R_{\Delta,\omega}(q^{1/4})=N_{\Delta_\mathrm{trop},\omega}^{\partial,\mathrm{trop}}\in\ZZ[\Lambda^2 N].$$
\item If $S\neq 0$, let $\langle\omega,P\rangle$ denotes the evaluation of $\omega$ on the exponents vectors for a Laurent polynomial $P\in\ZZ[\Lambda^2 N]$, then
$$R_{\Delta,\omega}(q^{1/4})=(q-q^{-1})^{-S} \left\langle \frac{\omega}{2},N_{\Delta_\mathrm{trop},\omega}^{\partial,\mathrm{trop}}\right\rangle. \in\ZZ[q^{\pm 1/2}].$$
\end{itemize}
\end{theom}
The $1/4$ is to account for the fact that the curves counted to define $R_{\Delta,\omega}$ are of degree $2\Delta$, while the tropical curves are of degree $\Delta$. Morally, they differ by the square map, which changes the quantum index by a value $4$.

\begin{rem}
In fact, as the tropical invariants were known before the corresponding real classical invariants, the definition of the signs is reverse-engineered: the orientations on the moduli space $\MM_{0,R,S}^\tau$ are chosen so that the sign they define close to the tropical limit coincide with the signs provided by the tropical multiplicity and the correspondence theorem. We then prove that these signs indeed provide an invariant.
\end{rem}

For now, the proof of invariance misses to include a non-generic choice of $\omega$. Such a choice provokes the appearance of \textit{walls} of new kind while checking the invariance. This is not surprising since on the tropical side, the genericity assumption on $\omega$ is already needed in \cite{blomme2020refined} to get a refined invariance with the refined multiplicity $m_\Gamma^q$. If $\omega$ is not chosen generically, one needs to consider the exponents $a_w\wedge b_w$ as living in some quotient $\Lambda^2 N/K_\Delta^\omega$, or at least evaluate the $2$-form $\omega$ used to define the enumerative problem. The same approach might provide invariance but remains to be checked.

\bigskip

The paper is organized as follows. First, we give a precise definition of the quantum index of an oriented real curve, and formulas for its computation close to the tropical limit and for real rational curves. Then, we describe the enumerative problem and the signs for which the refined signed count of solutions provides an invariant result. We then state the invariance result and its relation to previously known tropical invariants. The next section is devoted to introduce the various ingredients needed to prove the results of the paper, including the orientations on the moduli space of curves that are used in the description of the signs. For details on the tropical geometric setting of the paper, which is not the main content of the present paper, the reader is referred to \cite{blomme2020computation}.

\medskip

\textit{Acknowledgments} The author is grateful to Hulya Arguz, Pierrick Bousseau and Travis Mandel for helpful discussions and insights on the use of refined invariants in related mathematical areas.

\section{Quantum index of a real curve}
\label{section quantum index}

\subsection{Setting and definition}

Let $f:\CC C\dashrightarrow N_{\CC^*}$ be a parametrized real curve of degree $\Delta=(n_k)\subset N$, with $N$ a lattice. We assume that $C$ is of type I, meaning that the real locus $\RR C$ splits $\CC C$ into two connected components $S$ and $\overline{S}$. The choice of a component $S$ induces an orientation of the real locus $\RR C$ as the boundary of $S$. We call $S$ an \textit{oriented real curve}.

\medskip

We have the logarithmic and the argument map defined on $N_{\CC^*}$:
$$\left\{ \begin{array}{l}
\log:N_{\CC^*}\rightarrow N_\RR , \\
\arg:N_{\CC^*}\dashrightarrow N\otimes\RR/2\pi\ZZ ,\\
\end{array} \right. $$
which apply the logarithm and the argument coordinate by coordinate. The image of $\CC C$ or $S$ under these maps are called respectively the \textit{amoeba} and \textit{coamoeba} of the curve $C$.\\

Let $\varpi$ be a $2$-form on the lattice $N$. It extends respectively to a $2$-form on $N_\RR$ and $N\otimes\RR/2\pi\ZZ$ which are also denoted by the letter $\varpi$.

\begin{defi}
The area of the amoeba and coamoeba with respect to $\varpi$ are defined as follows:
$$\mathcal{A}_{\log}(S,\varpi)=\int_{S} (\log\circ f)^*\varpi \text{ and }\mathcal{A}_{\arg} (S,\varpi)=\int_{S} (\mathrm{arg}\circ f)^*\varpi.$$
\end{defi}

Using the holomorphic $2$-form on $N_{\CC^*}$ that is also defined by $\varpi$, it is classical to show that
$$\mathcal{A}_{\log}(S,\varpi)=\mathcal{A}_{\arg}(S,\varpi).$$
Moreover, as the conjugation on $\CC C$ exchanges the components $S$ and $\overline{S}$ and reverses the orientation, one has
$$\mathcal{A}_{\log}(\overline{S},\varpi)=-\mathcal{A}_{\log}(S,\varpi).$$

\begin{rem}
In this paper, we consider two $2$-forms $\omega$ and $\varpi$ which might have to coincide. The $2$-form $\omega$ is used to define the enumerative problem and the $\varpi$ is chosen and used to compute the log-area of the amoeba or coamoeba.
\end{rem}

We now define the quantum indices used to count real oriented curves solution to the $\omega$-problem. There are several definitions because the refinement that we use to count real oriented curves depends on $\omega$:
\begin{itemize}[label=$\circ$]
\item The first and most refined quantum index is some bivector in $\Lambda^2 N$. It is solely defined for oriented real curves having only real intersection with the toric boundary. It provides an invariant in the $\omega$-problem when $\omega$ is chosen generically. It enables the computation of the log-area of the curve for any choice of $2$-form $\varpi$.
\item The second quantum index is just a shift of the log-area computed with the $2$-form $\varpi=\omega$, it provides an invariant for the $\omega$-problem in presence of complex constraints.
\end{itemize}

\begin{theo}\label{theorem quantum index}
Let $S$ be a real oriented curve of degree $\Delta$.
\begin{itemize}[label=$\circ$]
\item If $S$ has only real intersection with the toric boundary, there exists a unique class $h(S)\in\Lambda^2 N$, called the quantum index, such that for any $\varpi$ one has
$$\mathcal{A}_{\log}(S,\varpi)=\frac{\pi^2}{2}\langle h(S),\varpi\rangle.$$
\item If $S$ has complex intersection with the toric boundary, let $\varepsilon_j\theta_j\pi$ be the argument of $\iota_{n_j}\omega$ evaluated at $p_j$ for some complex puncture $p_j$. Then there exists some half-integer $k(S,\omega)\in\frac{1}{2}\ZZ$ such that
$$\mathcal{A}_{\log}(S,\omega)-\pi^2\sum_j \varepsilon_j(2\theta_j-1)=\pi^2 k(S,\omega).$$
\end{itemize}
\end{theo}

When $\omega$ and $\Delta$ and the context are clear and specified, we drop them out of the notation and write just $k(S)$, which can then denotes $h(S)$, or the half-integer $k(S,\omega)$.

\medskip

The proof of existence is an analogue of the proof of Proposition $3$ in \cite{mikhalkin2017quantum} but in a higher dimensional version.

\begin{proof}
We have the coamoeba map that induces a chain $\arg\circ f:S\rightarrow N\otimes\RR/2\pi\ZZ$. The boundary of this chain consists of (half-)geodesics which are in correspondence with the intersection points between $S$ and the toric boundary $\CC\Delta$. The half-geodesics correspond to real intersection points and full-geodesics to complex intersection points. If we compose with the reduction modulo $\pi$ (instead of $2\pi$) we still get a cycle composed this time of full geodesics, maybe traveled several times.

\medskip

If the intersection points are all real, they are fixed by the complex conjugation. Thus, the geodesics on the boundary of the chain pass through fixed points of $-\mathrm{id}$ on $N\otimes\RR/\pi\ZZ$. It follows that the image of the chain in the quotient by $\pm\mathrm{id}$ has no boundary, just as in \cite{mikhalkin2017quantum}, and we get a cycle in the space $K_N=\left(N\otimes\RR/\pi\ZZ\right)/\{\pm\mathrm{id}\}$. This cycle realizes an integer homology class in $H_2(K_N,\ZZ)$. Meanwhile, as $-\mathrm{id}$ acts trivially on the second cohomology group, any $2$-form $\varpi$ descends into a cohomology class in $H^2(K_N,\ZZ)$. The result follows.

\medskip

If there are some complex intersection points, let $p_j^+$ be such a puncture. The corresponding geodesic does not necessarily pass through a fixed point of the complex conjugation. One thus needs to shift the geodesic so that it does so. Such a translation is not unique, and that is why we only consider the log-area for the choice $\varpi=\omega$. One translates the geodesic so that it passes through some fixed point, and the evaluation of $\iota_{n_j}\omega:N\otimes\RR/\pi\ZZ\rightarrow \RR/\pi\ZZ$ is now $\frac{\pi}{2}$. This shift modifies the log-area with respect to $\omega$ by a term $\pi^2\varepsilon_j(2\theta_j-1)$, and the log-area with respect to $\omega$ does not depend on the choice of the shift as long as its image by $\iota_{n_j}\omega$ is $\frac{\pi}{2}$.
\end{proof}

\begin{rem}
It is not possible to define a refined class in $\Lambda^2 N$ in presence of complex punctures, because contrarily to the planar case, there is no canonical direction in which to shift the complex geodesics to obtain a homology class in the quotient of the argument torus.
\end{rem}

\begin{rem}
If all intersection points with the toric boundary are real, the quantum index can appear as a generalization of the logarithmic area since in a way it computes the logarithmic area for any area form $\varpi$. It is possible to refine even more the quantum index by considering the homology class realized by $S$ in $H_2(N\otimes\RR/\pi\ZZ/\{\pm\mathrm{id}\},\ZZ)$, rather than the morphism it induces on $H^2(N\otimes\RR/\pi\ZZ,\ZZ)$. Concretely, it means not to consider the class up to torsion elements but to take care of this torsion part.
\end{rem}

\subsection{First properties}

With only real intersection points, due to its definition as a homology class in some quotient of the argument torus, the quantum index is well-behaved by transformations induced by monomial maps. If there are some complex punctures, the shift that we have to do to get the quantum index from the log-area is not wel-behaved by monomial maps, but the log-area is.

\begin{prop}
Let $\alpha:N'_{\CC^*}\rightarrow N_{\CC^*}$ be the monomial map associated to an integer matrix $A:N'\rightarrow N$, and let $f:S\dashrightarrow N'_{\CC^*}$ be an oriented real curve of degree $\Delta'$. Then, $\alpha\circ f:S\dashrightarrow N_{\CC^*}$ is a real oriented curve of degree $\Delta=A(\Delta')$. Assume $\omega'=A^*\omega$. Then we have respectively
\begin{itemize}[label=$\circ$]
\item With only real intersections,
$$h(\alpha(S))=(\Lambda^2 A)(h(S))\in\Lambda^2 N.$$
\item In presence of complex punctures, we only have
$$\mathcal{A}_{\log}(\alpha(S),\varpi)=\mathcal{A}_{\log}(S,A^*\varpi).$$
\end{itemize}
\end{prop}

\begin{rem}
The last statement is a generalization of the result from \cite{blomme2020computation}. The difference is that in \cite{blomme2020computation} $\omega$ and $\omega'$ were chosen to be generators of $\Lambda^2 M$, and thus the relation between $\omega$ and $\omega'$ was rather $A^*\omega=(\det A)\omega'$.
\end{rem}

\begin{proof}
It follows from the equalities:
\begin{align*}
\mathcal{A}_{\log}(\alpha(S),\varpi) &  = \int_{S}(\log\circ\alpha\circ f)^*\varpi \\
& = \int_{S}(A\circ\log\circ f)^*\varpi \\
& = \int_{S}(\log\circ f)^*A^*\varpi.\\
\end{align*}
If there are only real punctures, then one has in particular,
$$\mathcal{A}_{\log}(\alpha(S),\varpi)=\frac{\pi^2}{2}\langle h(S),A^*\varpi\rangle=\frac{\pi^2}{2}\langle (\Lambda^2 A)h(S),\varpi\rangle.$$
\end{proof}

Just as in the planar case from \cite{mikhalkin2017quantum}, it is possible to compute the quantum index of curves in a family close to its tropical limit. The additive nature of the logarithmic area ensures that the quantum index close to the tropical limit splits itself as a sum over the vertices of the tropical curve. For more details on the tropical limit, see \cite{blomme2020computation}.

\begin{prop}
Let $f_t:\CC C_t\dashrightarrow N_{\CC((t))^*}$ be a family of type I real curves, a choice of orientation $S_t$, and converging to a tropical limit $h:\Gamma\rightarrow N_\RR$. For each real vertex $w$ of $\Gamma$, let $f_w:\CC C_w\dashrightarrow N_{\CC^*}$ be the associated real curve with an induced orientation $S_w$. Then close to the tropical limit, one has
$$\mathcal{A}_\mathrm{log}(S_t,\varpi)=\sum_w \mathcal{A}_\mathrm{log}(S_w,\varpi),$$
where the sum is over the real vertices of $\Gamma$.
\end{prop}

\begin{rem}
Recall that the real vertices of $\Gamma$ are the vertices which are fixed by the involution on $\Gamma$ induced by the complex conjugation, \textit{i.e.} the vertices corresponding to real components in the nodal fiber of a stable model of $\CC C_t$. These vertices might not always correspond to curves having only real intersection with the toric boundary, and that is why we may not speak freely of their quantum index. It is the case if the curves in the family are rational curves having real intersection points with the toric boundary.
\end{rem}

\begin{proof}
The quantum index is continuous and takes discrete values, it is thus constant near the tropical limit. The additivity of the integral yields the result.
\end{proof}

\subsection{Computation}

\subsubsection{For toric type I curves}

We can compute the quantum index for any real curve that is of \textit{toric} type I. These are curves for which the image of $H_1(S,\ZZ)$ inside $H_1(N_{\CC^*},\ZZ)\simeq N$ is $\{0\}$. For more details, see \cite{mikhalkin2017quantum}. In particular, those curves have only real intersection points with the toric boundary of $\CC\Delta$.

\begin{theo}\label{theorem computation toric type I}
Let $f:\CC C\dashrightarrow N_{\CC^*}$ be a toric type I real curve with $S\subset\CC C$ a choice of orientation. Let $(\RR C^{(k)})_k$ be the collection of the real components of the real part $\RR C$, and for each $k$ let $(n_i^{(k)})_i$ be the collection of cocharacters associated to the boundary points on $\RR C^{(k)}$ in the cyclic order induced by the orientation on the curve. Then, for each $k$, $\sum_i n^{(k)}_i=0$, and one has
$$h(S)=\sum_k \sum_{i<j} n^{(k)}_i\wedge n^{(k)}_j,\in\Lambda^2 N$$
where $\sum_{i<j}$ is made choosing any starting point to make the cyclic order into a true order.
\end{theo}

\begin{proof}
The relation $\sum_i n^{(k)}_i=0$ immediately follows from the facts that the morphism
$$f_*:H_1(S)\rightarrow H_1(N_{\CC^*},\ZZ)\simeq N,$$
is zero, and that $\sum_i n^{(k)}_i$ is the image of the class realized by $\RR C^{(k)}$ inside $S$. This ensures that the value of $\sum_{i<j} n^{(k)}_i\wedge n^{(k)}_j$ does not depend on the chosen starting point.
\medskip
To compute the quantum index, we use the fact that the triviality of the morphism in homology groups implies that the morphism between fundamental groups $\pi_1(S)\rightarrow \pi_1(N_{\CC^*})$ is also trivial, since $\pi_1(N_{\CC^*})\simeq H_1(N_{\CC^*})$. In particular, this allows us to lift the map $\left. 2\arg \right|_S:S\rightarrow N\otimes\RR/\pi\ZZ$ to its universal cover $N_\RR$. We get a map $\tilde{f}:S\rightarrow N_\RR$. The image $\tilde{f}(S)$ is some surface whose boundary consists of some segments linking points from the lattice $\pi N$ in $N_\RR$.
\medskip
Then, for any choice of $\varpi$, as $\varpi$ is exact on $N_\RR$, the integral $\int_{\tilde{f}(S)}\varpi$ does only depend on the boundary of $\tilde{f}(S)$. Then, one can equivalently integrate a primitive of $\varpi$ on $\partial\tilde{f}(S)$ or choose another surface with boundary $\partial\left(\tilde{f}(S)\right)$, \textit{e.g.} triangulating the polygons $\tilde{f}(\RR C^{(k)})$, to yield the formula. (See \cite{blomme2020phd} or \cite{blomme2020computation}.)
\end{proof}

As disks are contractible, the rational curves are examples of toric type I curves if they have only real intersection points with the toric boundary, and their quantum index is computed by Theorem \ref{theorem computation toric type I}. The Theorem also allows for the following immediate corollary, which enables for the computation of the quantum index near the tropical limit, provided that the limit tropical curve is trivalent, and has trivial real structure.

\begin{coro}
For a real oriented rational curve with three boundary points, with associated cocharacters $n_1,n_2,n_3$, indexed in the cyclic order induced by the orientation, one has
$$h(S)=n_1\wedge n_2=n_2\wedge n_3=n_3\wedge n_1.$$
\end{coro}

Moreover, for rational curves with only real intersection points with the boundary, we have the following remarkable property.

\begin{coro}
For an oriented real rational curve having only real intersection points with the toric boundary, the quantum index does only depend on the order in which the curve intersects the toric divisors.
\end{coro}

\subsubsection{For rational curves}

For a rational curve with complex punctures, the computation from \cite{blomme2020computation} remains valid in higher dimension for the computation of the log-area with the $2$-form $\omega$, yielding the following result. Let us be given a parametrized oriented real rational curve
$$f:y\in\CC P^1\longmapsto \chi\prod_1^r (y-\alpha_i)^{n_i}\prod_1^s (y^2-2y\re\beta_j +|\beta_j|^2)^{n'_j}\in N_{\CC^*}.$$

\begin{theo}
\label{log-area rational curve}
If $r\geqslant 1$, assume that $\alpha_r=\infty$, and let $S$ be the upper hemisphere of $\CC P^1$. The logarithmic area is given by
$$\begin{array}{rcl}
\mathcal{A}_{\log}(S,\omega)= & & \sum_{i<i'}\omega(n_i,n_{i'})\frac{\pi^2}{2}\mathrm{sgn}(\alpha_{i'}-\alpha_i) \\
& + & \sum_{i,j}\omega(n_i,n'_{j'})2\pi\arctan\left( \frac{\alpha_i-\Re\beta_j}{\Im\beta_j} \right) \\
	& + & \sum_{j<j'}\omega(n'_j,n'_j)4\pi\arctan\left( \frac{\Re\beta_{j'}-\Re\beta_j}{\Im\beta_{j'}+\Im\beta_j} \right).\\
\end{array}$$
\end{theo}

\begin{proof}
For a complete proof, the reader is referred to \cite{blomme2020computation}. The idea is as follows:
\begin{itemize}[label=$\bullet$]
\item First use a monomial map $\alpha:(\CC^*)^{r+s}\rightarrow N_{\CC^*}$ to write $C$ as the image of some real oriented curve of degree $\{e_0,e_1,\dots,e_r,(e'_1)^2,\dots,(e'_s)^2\}$, where $|\Delta|-1=r+2s$ and $e_0=-\sum_1^r e_i-2\sum_1^s e'_i$.
\item Then, decompose $\alpha^*\omega$ in the basis $(e_i^*\wedge e_j^*)$ of $\Lambda^2(\ZZ^{r+s})^*$.
\item Using the projections $\ZZ^{r+s}\rightarrow\ZZ^2$ for each pair of coordinates, one reduces the computation of the log-area to the case of a line in $\CC P^2$, a parabola with complex intersection points, and a conic in $\CC P^2$ tangent to one axis and with complex intersection points with the other two axis. Those computations are made in \cite{blomme2020computation}.
\end{itemize}

\end{proof}

\section{$\omega$-problem and refined count of real oriented curves}

\subsection{Enumerative problem}

We consider the hereby described enumerative problem. Let $\Delta=(n_k,n_l^\pm)_{0\leqslant k\leqslant R, 1\leqslant l\leqslant S}\subset N$ be a degree consisting of $1+R+2S$ primitive vectors, where $n_l^+=n_l^-$, and let $\omega$ be a $2$-form on $N$ which is chosen generically, see section \ref{generic choice omega} or remark \ref{remark generic choice} for a precise statement. In particular, the vectors $n_k^{(\pm)}$ do not belong to $\ker\omega$. We look for real oriented rational curves of degree $2\Delta$, inside the positive orthant. In particular, these curves are tangent to the toric divisors. Such a curve has a parametrization of the form
$$f:y\longmapsto\chi\prod_{i=1}^R (y-\lambda_i)^{2n_i}\prod_{j=1}^S \left[(y-\lambda_j^+)(y-\lambda_j^-)\right]^{2n_j^+}\in N_{\CC^*},$$
where $\chi\in N_{\RR_+^*}$ is some positive cocharacter, and $\lambda_i^{(\pm)}$ are the coordinates of the intersection points with the toric boundary for some choice of coordinate on $\CC P^1$. From now on, we will mix up speaking about the real part of the curve, and its image by the logarithmic map $N_{\RR_+^*}\rightarrow N_\RR$, which is a diffeomorphism from the positive orthant to its image. In this new setting, with $\xi=\log\chi\in N_\RR$, the parametrization becomes
$$f:y\longmapsto\xi+\sum_{i=1}^R 2n_i \log|y-\lambda_i|+\sum_{j=1}^S 2n_j^+\log\left[(y-\lambda_i^+)(y-\lambda_i^-)\right]\in N_\RR.$$
This form allows for easier computations. Let $\MM_{0,R,S}^\tau$ be the moduli space of real oriented rational curves with $R+1$ real marked points $p_0,\dots,p_R$ and $S$ pairs of complex conjugated points $p_1^\pm,\dots,p_S^\pm$. The space of parametrized oriented real rational curves of degree $2\Delta$ tangent to the toric divisors in $\CC\Delta$ is in bijection with $N_\RR\times\MM_{0,R,S}^\tau$. The coordinate in the $N_\RR$ factor corresponds to the choice of $\xi$ once a coordinate on the curve has been fixed.

\begin{defi}
We have the \textit{moment map}
$$\mom:(\xi,(\lambda_i))\in N_\RR\times\MM_{0,R,S}^\tau\longmapsto (\mu_k),(\mu_k^+)\in \RR^{R}\times(\CC/2i\pi\ZZ)^S,$$
that sends a parametrized curve to the \textit{moment} of the curve at the real marked points $p_k$ for $k\neq 0$ and complex marked points $p_k^+$:
\begin{align*}
\mu_k^{(+)}= & \omega(n_k^{(\pm)},\xi)+\sum_{i=1}^R 2\omega(n_k^{(\pm)},n_{\gamma(i)})\log(\lambda(p_k^{(+)})-\lambda(p_{\gamma(i)})) \\
 & +\sum_{j=1}^S 2\omega(n_k,n^\pm_j)\log\left[(\lambda(p_k^{(+)})-\lambda(p_j^+))(\lambda(p_k^{(+)})-\lambda(p_j^-))\right],\\
\end{align*}
\textit{i.e.} the evaluation of the monomial $\iota_{n_k^{(+)}}\omega$ at $f(p_k^{(+)})$.
\end{defi}

For a complex marked point, the evaluation of the monomial is a complex number. Thus, its logarithm lives in $\CC/2i\pi\ZZ$.

\medskip

As $N_\RR\times\MM_{0,R,S}^\tau$ is of dimension $\rkn +R+2S-2$, if $\rkn\neq 2$, imposing the value of the moments is not enough to ensure having a finite number of curves. Thus, we choose a family of $\rkn-2$ monomials $m_1,\dots,m_{\rkn-2}\in M$ that completes $\iota_{n_0}\omega$ into a basis of $\gen{n_0}^\perp\subset M$, and evaluate them at $f(p_0)$.

\begin{defi}
We have the evaluation map
$$\ev: N_\RR\times\MM_{0,R,S}^\tau\longrightarrow \RR^{\rkn-2}\times\RR^R\times(\CC/2i\pi\ZZ)^S\equiv X,$$
that sends a curve to the evaluation of the monomials $m_i$ at $p_0$, and its moments $\mu_k^{(+)}$ at $p_k^{(+)}$.
\end{defi}

\begin{rem}
Notice that the Menelaus relation between the moments of the boundary points ensures that fixing the moments of every $p_k^{(+)}$ also fixes the moment of $p_0$.
\end{rem}

For a choice of $(P,\mu)\in\RR^{\rkn-2}\times\RR^R\times(\CC/2i\pi\ZZ)^S$, the enumerative problem consists in finding the curves that are sent to $(P,\mu)$ by the evaluation map. This amounts to find curves whose boundary points $p_k^{(+)}$ are sent to some toric suborbits prescribed by $\omega$, and $p_0$ is sent to some fixed point inside the suborbit defined by the others using the Menelaus condition.

\subsection{Invariance results}
\label{section statement of results}

By definition, all the curves solutions to the enumerative problem have a well-defined quantum index. Recall that if $S=0$, all the intersection points with the toric boundary are real and one can consider the quantum class $h(S)$, and if $S\neq 0$ one considers the log-area $k(S,\omega)$ with respect to the $2$-form $\omega$ used in the enumerative problem. This allows for a refined count of the solutions. However, in order to get a refined count that does not depend on the choice of the constraints, we also need to count the solutions with a sign that we now describe.

\medskip

Let $\oo(X)$ be an orientation on the space $X=\RR^{\rkn-2}\times\RR^R\times(\CC/2i\pi\ZZ)^S$ and let $\oo(N)$ be an orientation on the lattice $N$, \textit{e.g.} the one induced by $\omega$ when it is non degenerate. Let $\oo^{\gamma,J}(\MM)$ be the orientation on $\MM(\gamma,J)\subset\MM_{0,R,S}^\tau$ described in section \ref{section orientation moduli space}. Briefly, the moduli space of oriented real rational curves $\MM_{0,R,S}^\tau$ is a disjoint union of its connected components $\MM(\gamma,J)$, indexed by $\gamma\in\mathcal{S}_R$, which are all the possible cyclic orders that the orientation induces on the real marked points, and the subsets $J\subset[\![1;S]\!]$ of indices such that $p_j^-\in S$. In other terms, for a permutation $\gamma\in\mathcal{S}_R$, $\MM(\gamma,J)$ stands for the component of curves such that the induced order is
$$p_0<p_{\gamma(1)}<p_{\gamma(2)}<\cdots <p_{\gamma(R)}<p_0 ,$$
and $p_j^-\in S$ for $j\in J$. Each $\MM(\gamma,J)$ is orientable and has a canonical orientation $\oo^{\gamma,J}(\MM)$ provided by some choice of coordinates:
$$\oo^{\gamma,J}(\MM) = \left\{\begin{array}{l}
(-1)^R\oo(\eta) \text{ if }S\neq 0,\\
\oo(\nu_{1,N}) \text{ else}. \\
\end{array}\right.$$
See Section \ref{section orientation moduli space} for a precise definition. Together with $\oo(N)$, we get an orientation of $N_\RR\times\MM_{0,R,S}^\tau$.

\begin{defi}
The sign of a parametrized oriented curve $S\in N_\RR\times\MM(\gamma,J)$ is
$$\sigma(S)=\varepsilon(\gamma)(-1)^{|J|}\left.\frac{\ev^*\oo(X)}{\oo(N)\oo^{\gamma,J}(M)}\right|_S,$$
\textit{i.e.} it is $\varepsilon(\gamma)(-1)^{|J|}$ if $\ev$ is compatible with the chosen orientations, and $-\varepsilon(\gamma)(-1)^{|J|}$ if $\ev$ reverses the orientation.
\end{defi}

\begin{rem}
In concrete terms, the sign can be computed as $\varepsilon(\gamma)(-1)^{|J|}$ times the sign of determinant of the Jacobian matrix of the evaluation map when basis of both tangent spaces are chosen in an oriented way. 
\end{rem}

Now, let $(P,\mu)\in\RR^{\rkn-2}\times\RR^N$ be a generic choice of constraints, and let 
$$R_{\Delta,\omega}(P,\mu)=\sum_{S\in\ev^{-1}(P,\mu)}\sigma(S)q^{k(S)}\in\ZZ[\Lambda^2 N] \text{ or }\ZZ[q^{\pm 1/2}],$$
be the refined signed count of solutions to the enumerative problem. This is a Laurent polynomial in $\bino{\rkn}{2}$ variables if $S=0$, and a Laurent polynomial in $1$ variable if $S\neq 0$.

\begin{rem}
In the case $S=0$, to get Laurent polynomial in one variable, one can evaluate a $2$-form on the exponents, for instance $\omega$. However, it is worth noticing that in this case we have a coarser invariant since the $2$-form that we evaluate on the exponents needs not to be $\omega$, as it is yet the case in \cite{mandel2015scattering}.
\end{rem}

We have the following invariance statement.

\begin{theo}\label{theorem invariance}
The value of $R_{\Delta,\omega}(P,\mu)$ does not depend on the choice of $(P,\mu)$ as long as it is generic.
\end{theo}

The refined count independent of $(P,\mu)$ is then denoted by $R_{\Delta,\omega}$. Notice that it does not only depend on the choice of $\omega$ and the degree $\Delta$, but also on the choice of the marked point which has additional constraints, here $p_0$, which corresponds to the first element of $\Delta$, and the number of complex intersections with each toric divisor.

\begin{rem}
These invariants generalize the planar refined invariants introduced by Mikhalkin in \cite{mikhalkin2017quantum}.
\end{rem}

\subsection{Tropical computation of the invariant}

As the value of $R_{\Delta,\omega}(P,\mu)$ does not depend on the choice of $(P,\mu)$, it is natural to try compute it by choosing $(P,\mu)$ close to the tropical limit. This allows us to relate $R_{\Delta,\omega}$ to some already known tropical invariants considered in \cite{blomme2020refined}. We briefly recall their definition and refer to \cite{blomme2020refined} or \cite{mandel2015scattering} for a more complete description.

\medskip

Let $\MM_0(\Delta_\mathrm{trop},N_\RR)$ be the moduli space of degree $\Delta_\mathrm{trop}=(n_k,2n_l^+)_{0\leqslant k\leqslant R,1\leqslant l\leqslant S}$ parametrized rational tropical curves in $N_\RR$, which has an evaluation map
$$\ev_\mathrm{trop}:\MM_0(\Delta_\mathrm{trop},N_\RR)\longrightarrow \RR^{\rkn-2}\times\RR^N,$$
that evaluates the monomials $m_i$ on the first unbounded end, and $\iota_{n_k}\omega$ on the other $N=R+S$ unbounded ends.

\begin{defi}
The refined multiplicity of a parametrized trivalent tropical curve $h:\Gamma\rightarrow N_\RR$ is
$$m_\Gamma^q=\prod_{w}(q^{a_w\wedge b_w}-q^{-a_w\wedge b_w})\in\ZZ[\Lambda^2 N],$$
where the product is over the vertices of $\Gamma$, and $a_w$ and $b_w$ are the slopes of two out of the three outgoing edges in an order such that $\omega(a_w,b_w)>0$.
\end{defi}

For a generic choice of $(P,\mu)\in\RR^{\rkn-2}\times\RR^N$, the tropical curves in $\ev_\mathrm{trop}^{-1}(P,\mu)$ are trivalent, and we can make the refined count
$$N^{\partial,\mathrm{trop}}_{\Delta,\omega}(P,\mu)=\sum_{\Gamma\in\ev_\mathrm{trop}^{-1}(P,\mu)} m_\Gamma^q\in\ZZ[\Lambda^2 N].$$
The resulting Laurent polynomial does in fact not depend on the choice of $(P,\mu)$ as long as it is generic. See \cite{blomme2020refined} or \cite{mandel2015scattering} for a proof.

\begin{rem}
Concretely, it means that we are counting rational tropical curves whose first unbounded end is fixed, and whose other unbounded ends are contained in some fixed affine hyperplane of slope $n_k^{\perp_\omega}$, where $n_k$ directs the corresponding unbounded end.
\end{rem}

\begin{theo}\label{theorem tropical computation}
One has the following:
\begin{itemize}[label=$\circ$]
\item If there are no complex constraints, \textit{i.e.} $S=0$, then
$$R_{\Delta,\omega}(q^{1/4})=N_{\Delta_\mathrm{trop},\omega}^{\partial,\mathrm{trop}}\in\ZZ[\Lambda^2 N].$$
\item If $S\neq 0$, let $\langle\omega,P\rangle$ denotes the evaluation of $\omega$ on the exponents vectors for a Laurent polynomial $P\in\ZZ[\Lambda^2 N]$, then
$$R_{\Delta,\omega}(q^{1/4})=(q-q^{-1})^{-S} \left\langle \frac{\omega}{2},N_{\Delta_\mathrm{trop},\omega}^{\partial,\mathrm{trop}}\right\rangle. \in\ZZ[q^{\pm 1/2}].$$
\end{itemize}
\end{theo}

\begin{proof}
We use the real version of the correspondence theorem proved in \cite{blomme2020computation}, the fact that the quantum index is computed as a sum over the vertices close to the tropical limit, and a consequence of Theorem \ref{prop orientations tropical limit} that asserts that the sign we use to count real curves agree with the tropical sign close to the tropical limit. The proof is readily the same as the proof from the planar case in \cite{blomme2020computation}.

\medskip

Here is a sketch of what happens in the totally real case:
\begin{itemize}[label=$\bullet$]
\item For each tropical solution $\Gamma$ with $N-1$ vertices, the correspondence theorem builds $2^{N-1}$ solutions to the enumerative problem close to the tropical limit, according to how we \textit{glue} the curves $C_w$ over the vertices of $\Gamma$ together.
\item Each \textit{gluing} is obtained by choosing at each vertex $w$ an orientation $S_w$ on the corresponding curve $C_w$, and then glue along the edges in the orientation preserving way.
\item As the quantum index close to the tropical limit is the sum of the quantum indices over the vertices $w$, and as $h(S_w)=\pm a_w\wedge b_w$, we can gather the contribution of the $2^{N-1}$ solutions into
$$\prod_w (q^{a_w\wedge b_w}-q^{-a_w\wedge b_w})=m^q_\Gamma,$$
provided that we count the oriented curves with a sign $\prod_w\mathrm{sgn}\ \omega(a_w,b_w)$.
\item Finally, Theorem \ref{prop orientations tropical limit} ensures that the sign $\prod_w\mathrm{sgn}\ \omega(a_w,b_w)$ corresponds to the $\sigma$ used in the classical refined count.
\end{itemize}
\end{proof}

\section{Ingredients for the proof}

We present in this section the main technical tools used in the proof of Theorem \ref{theorem invariance} and Theorem \ref{theorem tropical computation}.

\subsection{Orientations on the moduli space of real oriented curves}
\label{section orientation moduli space}
%orientations and proof they coincide with the tropical signs

Consider the moduli space $\MM_{0,R,S}^\tau$ of oriented real rational curves with $1+R+2S$ marked points $p_0,\dots,p_R,p_{1}^\pm,\dots,p_{S}^\pm$. There are $R+1$ real points $p_i$ and $S$ pairs of complex conjugated points $p_j^\pm$. The connected components of $\MM_{0,R,S}^\tau$ are in bijection with the permutation $\gamma$ of $[\![1;R]\!]$ induced by the cyclic order on the real marked points, along with the repartition of the marked points $p_j^\pm$ on the two hemispheres.

\begin{defi}
Let $\MM(\gamma,J)$ be the component of $\MM_{0,R,S}^\tau$ corresponding to curves whose cyclic ordering on the real points $p_j$ is
$$p_0<p_{\gamma(1)}<p_{\gamma(2)}<\cdots < p_{\gamma(R)}<p_0,$$
and with $J=\{j \text{ s.t. }p_j^-\in S\}$, where $S$ denotes the component inducing the orientation of the curve.
\end{defi}

Notice that choosing a coordinate for which $p_0$ is $\infty$, there is a well-defined projection of the complex points $p_j^\pm$ on the real axis: $p_j^\pm\mapsto\re p_j$. For $i\in[\![0;R]\!]$, let $q_i=p_i$, and for $j\in[\![1;S]\!]$, let $q_{R+j}=\re p_j^\pm$. In that case, let $q_{R+j}^\pm=p_j^\pm$. We then have a family of real marked points $(q_j)_{0\leqslant j\leqslant N}$, where $N=R+S$. If we restrict to the open set $U$ of $\MM(\gamma,J)$ where the $q_j$ are distincts, the connected components of $U\subset\MM(\gamma,J)$ are the $\MM(\tilde{\gamma},J)$, indexed by the possible cyclic orders on the points $q_j$. Those are in bijection with the set of permutations $\tilde{\gamma}$ on $[\![1;N]\!]$ which induce the permutation $\gamma$ on the real marked points $p_i$ since the order induced on the real markings is fixed by $\MM(\gamma,J)$. Only the position of the $(q_j)_{R<j\leqslant N}$ associated to complex marked points might change:
$$p_0=q_0<q_{\tilde{\gamma}(1)}<\cdots<q_{\tilde{\gamma}(N)}<q_0.$$

We now specify some orientations on both $\MM(\gamma,J)$ and $\MM(\tilde{\gamma},J)$.

\begin{itemize}[label=$\bullet$]
\item For $\MM(\gamma,J)$, assuming $S\neq 0$, let $p_t^\pm$ be a pair of complex marked points. Let $\eta_t$ be the unique coordinate on $\CC P^1$ preserving the orientation such that $\eta_t(p_0)=\infty$ and $\{\eta_t(p_t^\pm)\}=\{\pm i\}$. Then, with $\sigma_j$ such that $p_j^{\sigma_j}\in S$, the data
$$\left( \eta_t(p_{\gamma(1)}),\dots,\eta_t(p_{\gamma(R)}),\eta_t(p_1^{\sigma_1}),\dots,\widehat{\eta_t(p_t^{\sigma_t})},\dots,\eta_t(p_S^{\sigma_S}) \right) \in\RR^{R}\times\CC^{S-1} ,$$
gives a set of coordinates on $\MM(\gamma,J)$ and a bijection to some open set of $\RR^R\times\CC^{S-1}$. As this space has a canonical orientation, it induces a natural orientation denoted by $\oo(\eta_t)$.

\item For $\MM(\tilde{\gamma},J)$ let $1\leqslant r<s\leqslant N$, so that $q_0< q_{\tilde{\gamma}(r)}<q_{\tilde{\gamma}(s)}<q_0$. Let $\nu_{r,s}$ be the unique coordinate such that
$$\left\{ \begin{array}{l}
\nu_{r,s}(q_0)=\infty,\\
\nu_{r,s}(q_{\tilde{\gamma}(r)})=0,\\
\nu_{r,s}(q_{\tilde{\gamma}(s)})=1.\\
\end{array}\right.$$
The data of every $\nu_{r,s}(q_{\tilde{\gamma}(j)})\in\RR$ associated to a real marking (resp. $\nu_{r,s}(q_{\tilde{\gamma}(j)}^{\sigma_{\tilde{\gamma}(j)}})\in\CC$ for complex markings) except $\nu_{r,s}(q_{\gamma(r)})$ and $\nu_{r,s}(q_{\gamma(s)})$ (equal to $0$ and $1$), in the order prescribed by $\tilde{\gamma}$ gives a set of coordinates on $\MM(\tilde{\gamma},J)$. We get a bijection to some open subset of
$$\prod_{i=1}^N\left\{\begin{array}{l}
\RR \text{ if }q_{\tilde{\gamma}(i)}\text{ is a real marking},i\neq r,s\\
\CC \text{ if }q_{\tilde{\gamma}(i)}\text{ is a complex marking},i\neq r,s\\
\{0\} \text{ if }q_{\tilde{\gamma}(i)}\text{ is a real marking},i= r,s\\
i\RR \text{ if }q_{\tilde{\gamma}(i)}\text{ is a complex marking},i= r,s\\
\end{array}\right. .$$
It also induces a natural orientation on the component, which is denoted by $\oo(\nu_{r,s})$. In other words, we remember the coordinates, real or complex, of every marked point in the order given by $\tilde{\gamma}$, except for $q_{\tilde{\gamma}(r)}$ and $q_{\tilde{\gamma}(r)}$, where we forget their coordinate if they are real, or their real part if they are complex.
\end{itemize}

\begin{rem}
If there are no complex marked points, the only coordinates are the $\nu_{r,s}$ since there are no $\eta_t$ coordinate.
\end{rem}

These various orientations of the components $\MM(\gamma,J)$ relate as follows.

\begin{prop}
One has the following:
\begin{enumerate}[label=(\roman*)]
\item The orientation $\oo(\eta_t)$ on $\MM(\gamma,J)$ does not depend on $t$: $\oo(\eta_t)=\oo(\eta_1)$, and is thus denoted $\oo(\eta)$.

\item On $\MM(\tilde{\gamma},J)$, let $N(r,s,\tilde{\gamma})$ be the number of coordinates between $\nu_{r,s}(q_{\tilde{\gamma}(r)})$ and $\nu_{r,s}(q_{\tilde{\gamma}(s)})$. It has the same parity as the number of real markings in the interval, plus one if $q_{\tilde{\gamma}(r)}$ is a complex marking (for its imaginary part). Then we have
$$\oo(\nu_{r,s}) =(-1)^{N(r,s,\tilde{\gamma})}\oo(\eta).$$

\item If there are no complex marked point, then $R=N$ and one has
$$\oo(\nu_{r,s})=(-1)^{N+r+s+1}\oo(\nu_{1,N}).$$
\end{enumerate}
\end{prop}

\begin{proof}
\begin{enumerate}[label=(\roman*)]
\item We have that $\eta_t=\frac{\eta_1-\re p_t^\pm}{\im p_t^{\sigma_1}}$. This allows us to write the Jacobian matrix associated to the change of coordinates. We check that it has positive determinant.
\item If $q_{\tilde{\gamma}(r)}$ or $q_{\tilde{\gamma}(s)}$ is a complex point $p_t^\pm$, we can write the Jacobian matrix of the change of coordinate from $\nu_{r,s}$ to the chosen coordinate $\eta_{\tilde{\gamma}(r)}$ or $\eta_{\tilde{\gamma}(s)}$ and check that it has determinant of sign $(-1)^{N(r,s,\tilde{\gamma})}$. If both points are real, the we relate to $\eta_1$ to get the same result.
\item If there are no complex marked points, we proceed similarly writing the Jacobian matrix of the change of coordinates.
\end{enumerate}
\end{proof}

\begin{rem}
Instead of remembering the coordinate of $p_t^{\sigma_t}\in S$, we could remember the coordinate of $p_t^{-\sigma_t}\in\overline{S}$. For each such point, this changes the orientation by $-1$ since it amounts to reverse the sign of one imaginary coordinates.
\end{rem}

\begin{defi}
We set
$$\oo^{\gamma,J}(\MM) = \left\{\begin{array}{l}
(-1)^R\oo(\eta) \text{ if }S\neq 0,\\
\oo(\nu_{1,N}) \text{ if }S=0. \\
\end{array}\right.$$
\end{defi}

\subsection{Coordinates near a strata of reducible curves}
\label{section coordinates reducible curves}

We assume that the degree $\Delta$ splits as a disjoint union $\Delta_A\sqcup\Delta_B$, with $n_0\in\Delta_A$, which are both degrees of respective sizes $p$ and $q$, so that $N=p+q$. Let $\MM(\gamma,J)$ be a component of $\MM_{0,R,S}^\tau$ adjacent to the strata of reducible curves with two components of respective degrees $\Delta_A$ and $\Delta_B$. A reducible curve is written $C_A\cup C_B$. The marked points, both real and pairs of conjugated points split themselves between the two components $C_A$ and $C_B$, meaning there are some real marked points $p_{\gamma(r)}<\cdots<p_{\gamma(s)}$ and complex marked points $p_j^\pm$ for $j\in K_B\subset[\![1;S]\!]$ that specialize to $C_B$. In other words, they have the same coordinate on $C_A$.

\medskip

We consider a subcomponent $\MM(\tilde{\gamma},J)$ of $\MM(\gamma,J)$ adjacent to the same strata of reducible curves. Confusing the notations for $\tilde{\gamma}(k)$ and $q_{\tilde{\gamma}(k)}$, the order $p_0<\tilde{\gamma}(1)<\cdots<\tilde{\gamma}(N)<p_0$ on the marked points can be rewritten
$$q_0<\tilde{\alpha}(1)<\cdots<\tilde{\alpha}(r)<\tilde{\beta}(1)<\cdots<\tilde{\beta}(q)<\tilde{\alpha}(r+1)\cdots<\tilde{\alpha}(p)<q_0,$$
for some $r$, where $\tilde{\alpha}(1),\dots,\tilde{\alpha}(p)$ are the marked points specializing to $C_A$ and $\tilde{\beta}(1),\dots,\tilde{\beta}(q)$ specializing to $C_B$. The $r$ corresponds to the position of the points specializing to $C_B$.

\begin{rem}
Notice that the order on the points $q_j$ belonging to $C_B$ is fixed using the projection away from $p_0$. If one forgets $p_0$, there is no such order anymore. Thus, the only things that are fixed are the cyclic order induced on the real markings $p_i\in B$, and the subset of indices $K_B\cap J$ of points such that $p_j^-$ is on the upper hemisphere.
\end{rem}

We now describe a specific set of coordinates on $\MM(\tilde{\gamma},J)$ that extends to a set of coordinates on the strata of reducible curves $C_A\sqcup C_B$.
\begin{itemize}[label=$\bullet$]
\item Let $A$ be the unique coordinate such that $A(q_0)=\infty$, $A(q_{\tilde{\alpha}(1)})=0$ and $A(q_{\tilde{\alpha}(p)})=1$. In other words it is equal to $\nu_{q+1,N}$ if $r=0$, $\nu_{1,p}$ if $r=p$, and $\nu_{1,N}$ else. It specializes on a coordinate on $C_A$.
\item Let $B$ be the unique coordinate such that $B(q_0)=\infty$, $B(q_{\tilde{\beta}(1)})=0$ and $B(q_{\tilde{\beta}(q)})=1$. It specializes to a coordinate on $C_B$ such that the intersection point between $C_A$ and $C_B$ has coordinate $\infty$. In other words, $B=\nu_{r+1,r+q}$.
\end{itemize}

Both coordinates are related by the following identity:
$$A=A(q_{\tilde{\beta}(1)})+(A(q_{\tilde{\beta}(q)})-A(q_{\tilde{\beta}(1)}))B.$$
We now take the following data as a set of coordinates on $\MM(\tilde{\gamma},J)$:
$$\left\{\begin{array}{l}
a_i=A(q_{\tilde{\alpha}(i)}) \text{ or }A(q_{\tilde{\alpha}(i)}^{\sigma_{\tilde{\alpha}(i)}}) \text{ for }i \neq 1,p, \\
\eta=A(q_{\tilde{\beta}(1)}), \\
\delta=A(q_{\tilde{\beta}(q)})-A(q_{\tilde{\beta}(1)}), \\
b_j=B(q_{\tilde{\beta}(j)}) \text{ or } B(q_{\tilde{\beta}(j)}^{\sigma_{\tilde{\beta}(j)}}) \text{ for }j\neq 1,q,\\
\end{array}\right.$$
taken in the order prescribed by $\tilde{\gamma}$:
$$(\widehat{a_1},\dots,a_r,\eta,b_2,\dots,b_{q-1},\delta,a_{r+1},\dots,\widehat{a_p}).$$
where the hat means removing (the real part of) the coordinate. The coordinate $\eta$ (resp. $\delta$) might be followed by the imaginary part of the $B$-coordinate $\im B(q_{\tilde{\beta}(1)}^{\sigma_{\tilde{\beta}(1)}})$ (resp. $\im B(q_{\tilde{\beta}(q)}^{\sigma_{\tilde{\beta}(q)}})$) of the point if $q_{\tilde{\beta}(1)}$ (resp. $q_{\tilde{\beta}(q)}$) is complex. The variable $\eta$ accounts for the position of the double point, \textit{i.e.} the position of the specialization of all $q_{\tilde{\beta}(j)}$ on $C_A$. The variable $\delta$ stands for the position of another point, namely $q_{\tilde{\beta}(s)}$, so that the knowledge of the $B(q_{\tilde{\beta}(j)})$ allows one to recover the position of the other marked points.

Unfortunately, in the above choice of coordinates, there are a few special cases to consider:
\begin{itemize}[label=$\circ$]
\item If $p=1$, which means that beside $p_0$, the only marked point left on $C_A$ is a real point or a pair of complex marked points, the coordinate $A$ is not properly defined. In that case, we define $A$ using $q_{\tilde{\alpha}(1)}$ and $q_{\tilde{\beta}(1)}$.
\item If $q=1$, as $C_B$ has at least three marked points, it means that $q_{\tilde{\beta}(1)}$ is associated to a pair of complex marked points $p_t^\pm$. We then take $B=\eta_t$. In that case,
$$A=\eta+\delta B=A(q_{\tilde{\beta}(1)})+\im A(q_{\tilde{\beta}(1)}^{\sigma_{\tilde{\beta}(1)}}))B.$$
Fortunately, this case does not occur in the proof since the component $C_B$ would have degree $0$.
\item Last, it is possible that some complex points specializing on $C_A$ belong to the interval $[\tilde{\beta}(1),\tilde{\beta}(q)]$, but this does not matter since it does not change the orientation.
\end{itemize}

\begin{prop}
\label{prop orientation reducible strata}
If $S\neq 0$, the orientation induced by the previous set of coordinates is
$$(-1)^{N(\tilde{\gamma},\tilde{\gamma}^{-1}\tilde{\alpha}(1),\tilde{\gamma}^{-1}\tilde{\alpha}(p))}\oo(\eta).$$
Moreover, in these coordinates, the strata of reducible curves is given by the equation $\delta=0$. Else, the orientation is $\oo(\nu_{1,p})$, $\oo(\nu_{q+1,N})$ or $\oo(\nu_{1,N})$.
\end{prop}

\begin{proof}
The second statement is immediate. To get the statement about the orientations, we consider the change of coordinates. The function that maps the new coordinates to the coordinate $A$, which induces the orientation $\oo(\nu_{1,p})$, $\oo(\nu_{q+1,N})$ or $\oo(\nu_{1,N})$ respectively, consists in replacing each $b$-coordinate by $\eta+\delta b$, and $\delta$ by $\delta+\eta$. We can check that it has positive determinant. Thus, the orientation is the same as the one induced by $A$, which is equal to $(-1)^{N(\tilde{\gamma},\tilde{\gamma}^{-1}\tilde{\alpha}(1),\tilde{\gamma}^{-1}\tilde{\alpha}(p))}\oo(\eta)$ if $S\neq 0$.
\end{proof}

\begin{rem}
Allowing $\delta$ to take negative values leads to coordinates on the component $\MM(\tilde{\gamma}',J')$, which is obtained by smoothing the node in the converse direction. In other words, with $\delta$ taking negative values, we get coordinates in the neighborhood of a reducible curve in the moduli space, and that involves both its possible smoothings, which live either in $\MM(\tilde{\gamma},J)$ or in $\MM(\tilde{\gamma}',J')$. On the other adjacent component $\MM(\tilde{\gamma}',J')$, where the order $\tilde{\gamma}'$ on the points is
$$q_0<\tilde{\alpha}(1)<\cdots<\tilde{\alpha}(r)<\tilde{\beta}(q)<\cdots<\tilde{\beta}(1)<\tilde{\alpha}(r+1)\cdots<\tilde{\alpha}(p)<q_0,$$
and one has $J'=J\Delta K_B\subset[\![ 1;S]\!]$: every complex point on $C_B$ changes side. When coming from this side, the coordinates that we get on the strata of reducible curves are
$$\left\{\begin{array}{l}
a'_i=A(q_{\tilde{\alpha}(i)}) \text{ or }A(q_{\tilde{\alpha}(i)}^{\sigma_{\tilde{\alpha}(i)}}) \text{ for }i \neq 1,p, \\
\eta'=A(q_{\tilde{\beta}(q)}), \\
\delta'=A(q_{\tilde{\beta}(1)})-A(q_{\tilde{\beta}(q)}), \\
b'_j=B'(q_{\tilde{\beta}(j)}) \text{ or } B'(q_{\tilde{\beta}(j)}^{-\sigma_{\tilde{\beta}(j)}}) \text{ for }j\neq 1,q,\\
\end{array}\right.$$ 
taken in the order
$$(\widehat{a'_1},\dots,a'_r,\eta',b'_{q-1},\dots,b'_2,\delta',a'_{r+1},\dots,\widehat{a_p}).$$
The coordinate $B'$ is related to $B$ via $B'=1-B$. Thus, one might check that the orientations defined on $\MM(\gamma,J)$ might not extend on the whole space, even with the twist by $\varepsilon(\gamma)$.
\end{rem}

\subsection{Crossing number of a reducible curve}

In this subsection, we still consider a reducible curve $C_A\cup C_B$ of degree $\Delta_A\sqcup\Delta_B$, which is limit of oriented curves, inducing an orientation on both components $C_A$ and $C_B$. Let $\tau_A$ and $\tau_B$ be the oriented tangent vectors to $C_A$ and $C_B$ at their common point. To keep an easy computation, we assume the intersection points with the toric boundary to be real, but the computation is readily the same with complex intersection points.

\begin{lem}
The oriented tangent vector to $C_A$ at the node is $\tau_A=\sum_{i=1}^p\frac{2n_{\alpha(i)}}{\eta-a_i}$ and the oriented tangent vector to $C_B$ at the node is $\tau_B=\sum_{j=1}^q 2n_{\beta(j)}b_j$.
\end{lem}

\begin{proof}
In the limit coordinate, $C_A$ is parametrized by
$$y\mapsto\xi+\sum_{i=1}^p 2n_{\alpha(i)}\log(y-a_i),$$
We just compute the derivative at $y=\eta$. Meanwhile, $C_B$ is parametrized as follows:
$$y\mapsto\xi+\sum_{i=1}^p 2n_{\alpha(i)} \log(\eta-a_i)+\sum_{j=1}^q 2n_{\beta(j)}\log(y-b_j).$$
As one cannot evaluate at $y=\infty$, we make the prior change of coordinate $z=-\frac{1}{y}$, which preserves the orientation, getting a new parametrization
$$z\mapsto\xi+\sum_{i=1}^p 2n_{\alpha(i)} \log(\eta-a_i)+\sum_{j=1}^q 2n_{\beta(j)}\log(1+b_j z),$$
yielding the result by derivating at $0$.
\end{proof}

We then have
$$\omega(\tau_A,\tau_B)=4\sum_{i,j}\omega(n_{\alpha(i)},n_{\beta(j)})\frac{b_j}{\eta-a_i}.$$
The Menelaus Theorem forces that $\sum_{k=1}^q \mu_{\beta(k)}=0$. In fact, the sign of the quantity $\omega(\tau_A,\tau_B)$, called the \textit{crossing number} of $C_A$ and $C_B$, determines whether a deformation of $C_A\cup C_B$ in the orientation preserving way is sent to $\{\sum_{k=1}^q \mu_{\beta(k)}>0\}$ or rather $\{\sum_{k=1}^q \mu_{\beta(k)}>0\}$ by the moment map.

\begin{prop}\label{prop which side of the wall}
If $\omega(\tau_A,\tau_B)>0$, then the smoothing in the orientation preserving way lies in $\{\sum_{k=1}^q \mu_{\beta(k)}<0\}$.
\end{prop}

\begin{proof}
The moment at a marked point $\beta(k)$ is
$$\mu_{\beta(k)} = \omega(n_{\beta(k)},\xi)
+ \sum_{i=1}^p 2\omega(n_{\beta(k)},n_{\alpha(i)}) \log(\eta+\delta b_k-a_i)
+ \sum_{j=1}^q 2\omega(n_{\beta(k)},n_{\beta(j)}) \log(b_k-b_j),$$
and thus
\begin{align*}
\sum_{k=1}^q\mu_{\beta(k)} & =\sum_{i,k} 2\omega(n_{\beta(k)},n_{\alpha(i)}) \log(\eta+\delta b_k-a_i)
+ \sum_{j,k} 2\omega(n_{\beta(k)},n_{\beta(j)}) \log(b_k-b_j) \\
 & =\sum_{i,k} 2\omega(n_{\beta(k)},n_{\alpha(i)}) \log(\eta+\delta b_k-a_i), \\
\end{align*}
since the term of the second sum is antisymmetric in $j$ and $k$. In particular, this function is constant equal to $0$ on $\{\delta=0\}$, as stated by the Menelaus Theorem. Then, the partial derivative with respect to $\delta$ evaluated at $\delta=0$ is
$$\sum_{i,k} 2\omega(n_{\beta(k)},n_{\alpha(i)})\frac{b_k}{\eta-a_i}=-\omega(\tau_A,\tau_B).$$
The result follows.
\end{proof}

\begin{rem}
In presence of complex intersection points, the sum $\sum\mu_{\beta(k)}$ is equal to the sum of the moments for the real intersection points, plus twice the sum of the common real part of the conjugated moments of points in a conjugated pair.
\end{rem}

\subsection{Orientations and tropical limit}

We consider the evaluation map restricted to some connected component $N_\RR\times\MM(\gamma,J)$ of $N_\RR\times\MM_{0,R,S}^\tau$:
$$\begin{array}{rccl}
\ev^{\gamma,J}: & N_\RR\times\MM(\gamma,J) & \longrightarrow & \RR^{\rkn-2}\times\RR^R\times(\CC/2i\pi\ZZ)^S=X ,\\
 & \left(\xi,(\lambda_k)\right) & \longmapsto & (\ev_{m_i})_{1\leqslant i\leqslant \rkn-2},(\mu_k)_{1\leqslant k\leqslant R},(\mu_k^+)_{1\leqslant k\leqslant S} \\
\end{array}$$
where
$$\mathrm{ev}_{m_i}=\langle m_i,\xi\rangle,$$
and
\begin{align*}
\mu_k^{(+)}= & \omega(n_k^{(\pm)},\xi)+\sum_{i=1}^R 2\omega(n_k^{(\pm)},n_{\gamma(i)})\log(\lambda(p_k^{(+)})-\lambda(p_{\gamma(i)})) \\
 & +\sum_{j=1}^S 2\omega(n_k,n^\pm_j)\log\left[(\lambda(p_k^{(+)})-\lambda(p_j^+))(\lambda(p_k^{(+)})-\lambda(p_j^-))\right].\\
\end{align*}
Notice that as usual, the complex logarithm has value in $\CC/2i\pi\ZZ$. This quotient disappears when we consider the derivative.

\medskip

For more details on the tropical setting, see \cite{blomme2020computation}. Let $C^{(t)}\in N_\RR\times\MM(\tilde{\gamma},J)$ be a family of oriented real rational curves inside one of the subcomponents of $\MM(\gamma,J)$. The tropical limit of this family is a parametrized real tropical curve $h:\Gamma\rightarrow N_\RR$, with a real structure denoted by $\sigma$. We assume that $\Gamma/\sigma$ is trivalent, and that $\Gamma$ has no flat vertex, which is the case for the tropical solutions admitting lift to solutions of the enumerative problem. By \textit{lift}, we mean that there exists some family of curves tropicalizing to it and satisfying the constraints of the enumerative problem. This is give by the realization theorem in \cite{blomme2020computation}. Thus, the vertices of $\Gamma$ are either trivalent, quadrivalent with a pair of exchanged edges, or pentavalent with two pairs of exchanged vertices.

\medskip

Furthermore, the real structure and the orientation on $C^{(t)}$ endow the quotient curve $\Gamma/\sigma$ with a \textit{ribbon structure}, \textit{i.e.} a cyclic order on the edges leaving each real vertex. At the totally real vertices, the ribbon structure is obtained as the orientation induced on the corresponding real component component. At the non-totally real vertices, the cyclic ordering is defined using the projection $p_j^\pm\mapsto\re p_j^\pm\in\RR P^1$.

\medskip

To each oriented real tropical curve $\Gamma$ is associated a sign obtained as follows:

\begin{defi}
Let $\mathcal{V}(\Gamma)$ be the set of vertices of $\Gamma/\sigma$ which come either from a real trivalent vertex, or from a real pentavalent vertex, and for such a vertex $W$, let $\omega_W=\omega(n_\mathrm{left},n_\mathrm{right})$ be the $2$-form evaluated at the slopes of two outgoing edges in the order prescribed by the ribbon structure. Then, we set
$$\sigma(\Gamma)=\prod_{W\in\mathcal{V}}\mathrm{sgn}\ \omega_W.$$
\end{defi}

We now show that near the tropical limit, $\sigma(\Gamma)$ and $\lim_t\sigma(C^{(t)})$ coincide. In other terms, we have a well-defined orientation on the component $\MM(\tilde{\gamma},J)$ so that the evaluation matrix has sign $\varepsilon(\gamma)(-1)^{|J|}\sigma(\Gamma)$ near the tropical limit $\Gamma$.

\medskip

The ends of the quotient tropical curve $\Gamma/\sigma$ are indexed by $[\![0;N]\!]$, in the order prescribed by $\tilde{\gamma}$. To each end of $\Gamma/\sigma$ corresponds a coordinate of the evaluation map. If $k$ corresponds to a real marking, we get $\mu_k\in\RR$, and if $k$ corresponds to a complex marking, we get $\mu_k^+\in\CC/2i\pi\ZZ$. The order in which we evaluate is fixed to be $p_1,\dots,p_R,p_1^+,\dots,p_S^+$. This order is imposed by the choice of an orientation on $X$. We can reorder into the order prescribed by $\tilde{\gamma}$, which changes the orientation by $\varepsilon(\gamma)$, since only the order on the real markings matters. Moreover, if we evaluate complex moments at $p_j^{\sigma_j}\in S$ rather than $p_j^+$, which changes the orientation by $(-1)^{|J|}$, since $J=\{j:p_j^- \in S\}$. Thus, we can assume that the coordinates are given by $\mu_k^{(\sigma_j)}$ in the order prescribed by $\tilde{\gamma}$, if we change the sign by $\varepsilon(\gamma)(-1)^{|J|}$.

\begin{figure}
\begin{center}
\definecolor{ffqqqq}{rgb}{1.,0.,0.}
\definecolor{xdxdff}{rgb}{0.49019607843137253,0.49019607843137253,1.}
\definecolor{ududff}{rgb}{0.30196078431372547,0.30196078431372547,1.}
\begin{tikzpicture}[line cap=round,line join=round,>=triangle 45,x=0.65cm,y=0.65cm]
\clip(0.,1.) rectangle (20.,14.);
\draw [line width=2.pt] (10.,13.)-- (10.,11.);
\draw [line width=2.pt] (10.,11.)-- (1.,2.);
\draw [line width=2.pt] (10.,11.)-- (19.,2.);
\draw [line width=2.pt] (3.,2.)-- (2.,3.);
\draw [line width=2.pt] (5.,6.)-- (9.,2.);
\draw [line width=2.pt] (7.,4.)-- (5.,2.);
\draw [line width=2.pt] (10.,2.)-- (14.53,6.47);
\draw [line width=2.pt] (12.,2.)-- (15.51,5.49);
\draw [line width=2.pt] (14.005697912670101,3.9942694345352288)-- (16.,2.);
\draw [line width=2.pt] (18.,3.)-- (17.,2.);
%\draw [<->,line width=2.pt,dash pattern=on 4pt off 4pt,color=ffqqqq] (9.,0.)-- (10.,0.);
%\draw [<->,line width=2.pt,dash pattern=on 4pt off 4pt,color=ffqqqq] (10.,0.)-- (12.,0.);
\draw [->,line width=2.pt] (5.,6.)-- (4.,5.);
\draw [->,line width=2.pt] (5.,6.)-- (6.,5.);
\begin{scriptsize}
\draw [fill=black] (10.,13.) circle (2.5pt);
\draw[color=black] (9.968299745797772,13.651309761832477) node {$\lambda_0$};
\draw [fill=ududff] (10.,11.) circle (2.5pt);
\draw[color=ududff] (11.5,11.240066934492399) node {$W_\mathrm{root}$};
\draw [fill=black] (1.,2.) circle (2.5pt);
\draw[color=black] (0.9578660225795702,1.3412805906752328) node {$\lambda_1$};
\draw [fill=black] (19.,2.) circle (2.5pt);
\draw[color=black] (18.957582216144573,1.573944372260679) node {$\lambda_N$};
\draw [fill=black] (3.,2.) circle (2.5pt);
\draw[color=black] (3.0729913097199932,1.3201293378038284) node {$\lambda_2$};
\draw [fill=xdxdff] (5.,6.) circle (2.5pt);
\draw[color=xdxdff] (6.,6.079161233869774) node {$W_3$};
\draw [fill=black] (9.,2.) circle (2.5pt);
\draw[color=black] (8.889585849356157,1.489339360775062) node {$\cdots$};
\draw [fill=black] (5.,2.) circle (2.5pt);
\draw[color=black] (4.913150309532161,1.4258856021608495) node {$\lambda_3$};
\draw [fill=black] (10.,2.) circle (2.5pt);
\draw[color=black] (9.925997240054965,1.5104906136464662) node {$\cdots$};
\draw [fill=black] (12.,2.) circle (2.5pt);
\draw[color=black] (11.914215009966963,1.5104906136464662) node {$\cdots$};
\draw [fill=black] (16.,2.) circle (2.5pt);
\draw[color=black] (15.932953055533766,1.489339360775062) node {$\cdots$};
\draw [fill=black] (17.,2.) circle (2.5pt);
\draw[color=black] (16.927061940489764,1.4681881079036578) node {$\lambda_{N-1}$};
\draw[color=black] (3.6017726315050993,5.740741187927307) node {$n_{W,\mathrm{left}}$};
\draw[color=black] (5.367902246267352,4.5) node {$n_{W,\mathrm{right}}$};
\end{scriptsize}
\end{tikzpicture}
\caption{}
\end{center}
\end{figure}

We root the tree $\Gamma/\sigma$ at $q_0$. We label the leaves $\lambda_1,\dots,\lambda_N$, so that $\lambda_l$ is associated to the marked point $q_{\tilde{\gamma}(l)}$. As $\Gamma/\sigma$ is trivalent and endowed with a ribbon structure, we can label its vertices $W_1,\dots,W_{N-1}$, so that $W_l$ is the nearest common ancestor of the leaves $\lambda_l$ and $\lambda_{l+1}$. Let $\kappa(l)$ be the minimum index among the leaves of $W_l$. Thus, the leaves descending from the first son of $W_l$ are indexed by $[\![\kappa(l);l]\!]$.

\medskip

To compute the determinant of the Jacobian of the evaluation map near the tropical limit, we use coordinates similar to the coordinates $(\alpha,\beta)$ from \cite{blomme2020computation}. Those coordinates are well-behaved near the tropical limit and enable an easy computation of the determinant. They are defined as follows. First, we define a coordinate $\nu_l=\nu_{W_l}$ associated to each vertex $W_l$ of $\Gamma/\sigma$.

\begin{itemize}[label=$\circ$]
\item If $W_l$ comes from a real vertex, let $\nu_l$ be the coordinate $\nu_{\kappa(l),l+1}$, which takes values $0$ and $1$ on the real part of the leaves $\lambda_{\kappa(l)}$ and $\lambda_{l+1}$.
\item If $W_l$ comes from a pair of complex trivalent vertices. Then $\nu_l$ is the coordinate such that $\nu_l(\lambda_{\kappa(l)}^{\sigma_{\kappa(l)}})=0$ and $\nu_l(\lambda_{l+1}^{\sigma_{l+1}})=1$.
\end{itemize}

For each pair of vertices $W_l,W_L$ linked by an edge $\gamma$ and with $W_l$ being the son of $W_L$, we have a relation:
$$\nu_L=\beta_\gamma+\alpha_\gamma\nu_l,$$
where:
\begin{itemize}[label=$\circ$]
\item If $W_l$ is a real vertex, $\alpha=\nu_L(\lambda_{l+1})-\nu_L(\lambda_{\kappa(l)})\in\RR$ and $\beta=\nu_L(\lambda_{\kappa(l)})\in\RR$.
\item If $W_l$ is a complex vertex, then $\alpha=\nu_L(\lambda_{l+1}^{\sigma_{l+1}})-\nu_L(\lambda_{\kappa(l)}^{\sigma_{\kappa(l)}})\in\CC$ and $\beta=\nu_L(\lambda_{\kappa(l)}^{\sigma_{\kappa(l)}})\in\CC$.
\end{itemize}

\medskip

Initially, as the vertex $W_\mathrm{root}$ is always real, the component $\MM(\tilde{\gamma},J)$ is endowed with the orientation $\nu_{1,l_\mathrm{root}+1}$ induced by the data of the coordinates of all points on the curve except (the real parts of) the coordinates of $\lambda_1$ and $\lambda_{l_\mathrm{root}+1}$.

\medskip

We now do some inductive process to transform this set of coordinates into the more suitable coordinates similar to \cite{blomme2020computation}. This is done by inductively replacing some $\nu_L(\lambda_k)$ by $\nu_{l}(\lambda_k)$ if $\lambda_k$ is not directly adjacent to $W_L$, and $W_l$ is the son of $W_L$ ancestor of $\lambda_k$.

\begin{itemize}[label=$\circ$]
\item If $W_l$ is real, then we have $\nu_L=\beta+\alpha\nu_l$ where
$$\left\{\begin{array}{l}
\beta=\nu_L(\lambda_{\kappa(l)}),\\
\alpha=\nu_L(\lambda_{l+1})-\nu_L(\lambda_{\kappa(l)}).\\
\end{array}\right. $$
Then, for each coordinate $\nu_L(\lambda_k)$ for $\lambda_k$ descendant of $W_l$, we replace it by $\nu_l(\lambda_k)$, except for $\nu_L(\lambda_{\kappa(l)})=\beta$ which stays, and $\nu_L(\lambda_{l+1})$ which is replaced by $\alpha$. Notice that the change includes the imaginary part of $\lambda_{\kappa(l)}$ and $\lambda_{l+1}$ when they are complex leaves. The function that maps new coordinates to previous coordinates is
$$(\beta,\nu_l(\lambda_k),\alpha,\nu_l(\lambda_k))\longmapsto (\beta,\beta+\alpha\nu_l(\lambda_k),\alpha+\beta,\beta+\alpha\nu_l(\lambda_k)),$$
where $\nu_l(\lambda_k)$ denotes the coordinates of the other leaves, including the potential imaginary parts of $\lambda_{\kappa(l)}$ and $\lambda_{l+1}$. In either case, this map preserves the orientation. Notice that (the real part of) the coordinate of $\lambda_{\kappa(l)}$ might not appear as a coordinate of the form $\nu_L(\lambda_{\kappa(l)})$ if the leaf has already been used to fix some $\nu_{l'}$.

\item If $W_l$ is a complex trivalent vertex, then for $\nu_L=\beta+\alpha\nu_l$, we have
$$\left\{ \begin{array}{l}
\alpha=\nu_L(\lambda_{l+1}^{\sigma_{l+1}})-\nu_L(\lambda_{\kappa(l)}^{\sigma_{\kappa(l)}}) ,\\
\beta=\nu_L(\lambda_{\kappa(l)}^{\sigma_{\kappa(l)}}).\\
\end{array} \right.$$
We replace each $\nu_L(\lambda_k)$ by $\nu_l(\lambda_k)$ except for $k=\kappa(l)$ where we rather take $\beta$, and $k=l+1$ where we take $\alpha$. The map that takes new coordinates to the previous one is
$$(\beta,\nu_{l}(\lambda_k),\alpha,\nu_{l}(\lambda_k))\longmapsto (\beta,\beta+\alpha\nu_{l}(\lambda_k),\alpha+\beta,\beta+\alpha\nu_{l}(\lambda_k)).$$
The vertex $W_L$ is either complex trivalent, in which case $\beta$ is already fixed to be $0$ or $1$ by the choice of $\nu_L$ and does not appear as a coordinate, or or a real quadrivalent or pentavalent vertex, in which case the real part of $\beta$ is already fixed to be $0$ or $1$ by the choice of $\nu_L$. Then, the coordinate in place of $\beta$ is either $\im\nu_L(\lambda_{\kappa(l)})$, or none. In either case, this change of coordinates preserves the orientation as it is a holomorphic map between complex coordinates.
\end{itemize}

The process terminates when a vertex $W_l$ is only adjacent to two leaves used to fix the coordinate $W_l$. In the end we get the following coordinates: 
\begin{itemize}[label=-]
\item for each vertex $W_l$ except the root, we get a coordinate $\alpha$, which is real and sits at the place of $\re\nu_\mathrm{root}(\lambda_{l+1})$ if the edge ending at $W_l$ is real, and complex and sits at the place of $\nu_\mathrm{root}(\lambda_{l+1})$ if it is complex.
\item a real coordinate $\beta$ for each quadrivalent vertex, equal to the imaginary part of the complex edge, at the place of $\im\nu_\mathrm{root}(\lambda_{\kappa(l)})$ or $\im\nu_\mathrm{root}(\lambda_{l+1})$ according to whether the complex edge is on the left or the right of $W_l$. (We speak of quadrivalent vertex on the \textit{right} or on the \textit{left}.)
\item a pair of real coordinates $(\beta_1,\beta_2)$ for each pentavalent vertex, equal to the imaginary parts of both complex edges, at the places of $\im\nu_\mathrm{root}(\lambda_{\kappa(l)})$ and $\im\nu_\mathrm{root}(\lambda_{l+1})$.
\end{itemize}

We can express the moments $\mu_k$ in the new coordinates $(\xi,\alpha,\beta)$, and compute the Jacobian matrix. Near the tropical limit, the $\beta$ coordinates converge to some finite positive value, while the $\alpha_\gamma$ coordinates vanish like $t^{|\gamma|}$, where $|\gamma|$ is the length of the edge associated to $\alpha$ in the tropical curve $\Gamma$. Therefore, we see the Jacobian matrix converges to the evaluation matrix of the tropical evaluation map on $\Gamma$. We now prove that the tropical sign agrees with the sign of the Jacobian matrix near the tropical limit.

\begin{theo}\label{prop orientations tropical limit}
Near the tropical limit, one has
$$\lim_t\sigma(C^{(t)})=\sigma(\Gamma).$$
\end{theo}

\begin{proof}
We use the lemmas that follow the proof of this Theorem to compute the determinant of the Jacobian matrix in the orientations $(-1)^{|J|}\varepsilon(\gamma)\oo(X)$ and $\oo(\nu_{1,l_\mathrm{root}+1})$. Each step allows us to prune a vertex, until we have just one vertex. Then, assuming that the basis of the lattice $N$ ends with $-n_0$, we can do a last pruning to get rid of the last vertex. Let $\sigma_0$ is the sign of the Gram determinant of the family $(m_1,\dots,m_{\rkn-2},\iota_{n_0}\omega)$ evaluated on the first $\rkn-1$ vectors of the basis of $N$. Viewing the column where we evaluate at $-n_0$ as the column of an $\alpha$ coordinate, one can again apply the following lemmas to do a last pruning. However, as the order is not the same as in the lemmas, we might get some extra minus sign as follows:
\begin{itemize}[label=$\circ$]
\item If $W_\mathrm{root}$ is trivalent, the last column of the evaluation matrix has only the last two rows with non-zero coefficients, which are $\omega_{W_\mathrm{root}}$ and $-\omega_{W_\mathrm{root}}$. We add the last row to the previous and develop to get $\omega_{W_\mathrm{root}}\sigma_0$. Since, $n_1+n_2=-n_0$, the sign changes when passing from $\iota_{n_1+n_2}\omega$ to $\iota_{n_0}\omega$.
\item If the vertex is quadrivalent on the right, we still have a $\beta$ coordinate, and the order is the same as in the corresponding lemma. Thus, we develop accordingly, and get $-\omega_{W_\mathrm{root}}^2\sigma_0$.
\item If the vertex is left-quadrivalent, we proceed similarly, but as the coordinates are in the reverse order regarding the order of the lemma, we get this time $\omega_{W_\mathrm{root}}^2\sigma_0$.
\item Finally, for a pentavalent vertex, we also need to permute two coordinates for getting back to the corresponding lemma, and we get $-\omega_{W_\mathrm{root}}\sigma_0$.
\end{itemize}
Hence, up to a positive scalar,
$$\det\ev(\Gamma)=\epsilon(-1)^{\sum_{\mathcal{V}(\Gamma)} 1+\hat{r}(\lambda_l)}\prod_{W\in\mathcal{V}(\Gamma)}\omega_W \sigma_0,$$
where $\epsilon$ is the sign depending on the nature of $W_\mathrm{root}$ and provided by the previous disjonction, equal to $-1$ if the second branch starts with a complex edge and $1$ else, and $\hat{r}(\lambda_l)$ is $1$ on the first branch and $2$ on the second one, comes from the lemmas.

\medskip

Furthermore, one has that $1+\hat{r}$ takes even value on the first branch, and odd on the second branch. Thus, the contribution of $\sum_{\mathcal{V}(\Gamma)} 1+\hat{r}$ to the global sign is equal to the number of trivalent and pentavalent vertices in the second branch. Doing a pruning, we see that this number is equal mod $2$ to the number $L_\mathrm{right}$ of leaves in the branch, minus $1$ if the branch does not start with a complex edge. Therefore, taking into account the sign $\epsilon$, equal to $-1$ when the branch starts with a complex edge, we get that $\epsilon(-1)^{\sum 1+\hat{r}}=(-1)^{L_\mathrm{right}+1}$ in any case.

\medskip

Meanwhile, if $S\neq 0$, we have that $\oo(\nu_{1,l_\mathrm{root}+1})=(-1)^{L_\mathrm{left}-1}\oo(\eta)$. As $L_\mathrm{left}+L_\mathrm{right}=R+2S$, in total, we get that
$$\varepsilon(\gamma)(-1)^{|J|}\left.\frac{\ev^*\oo(X)}{\oo(N)\oo(\eta)}\right|_\Gamma = (-1)^{R+2S}\sigma(\Gamma)\sigma_0.$$
Thus, we get the result. If $S=0$ and there is no complex marked point, we have
$$\oo(\nu_{1,l_\mathrm{root}+1})=(-1)^{L_\mathrm{left}-1}(-1)^N\oo(\nu_{1,N}),$$
and we still get the result as $L_\mathrm{left}+L_\mathrm{right}=N$.
\end{proof}

We now get to the lemmas enabling the pruning. For each vertex $W$, the parametrization of the curve writes itself
$$f_W:y\longmapsto \xi_W+\sum_{k\in\mathcal{F}(W)} 2n_k\log(y-\nu_W(\lambda_k))
+\sum_{k\notin\mathcal{F}(W)} 2n_k\log\left(\frac{y}{\nu_W(\lambda_k)}-1\right),$$
where $\mathcal{F}(W)$ denotes the sons of $W$ in $\Gamma$, whether they are real or complex. If $W$ is the son of $V$, and $\gamma$ is the edge between them, using $\nu_V=\beta+\alpha\nu_W$, following \cite{blomme2020computation} or \cite{tyomkin2017enumeration}, we get that the relation between $\xi_V$ and $\xi_W$ is
$$\xi_W=\xi_V+2n_\gamma\log\alpha_\gamma+\sum_{\mathcal{F}(V)\backslash \mathcal{F}(W)}2n_k\log(\nu_V(\lambda_k)-\beta_\gamma)+\sum_{k\notin\mathcal{F}(V)}2n_k\log\left(1-\frac{\beta_\gamma}{\nu_V(\lambda_k)}\right),$$
We see that at first order, the second sum vanishes. In particular, we see that a coordinate $\alpha_\gamma$ only interfers in the moment of the leaves in its branch, and a $\beta$ coordinate only interfers in the moments for the sons of its associated vertex. Near the tropical limit, at first order, we get the following partial derivatives:
\begin{itemize}[label=$\bullet$]
\item For a real end and real $\alpha$,
$$\frac{\partial\mu_k}{\partial\log\alpha_\gamma} =2\omega(n_k,n_\gamma) .$$
\item For a complex end and real $\alpha$,
$$\frac{\partial\mu_k}{\partial\log\alpha_\gamma} =\begin{pmatrix}
2\omega(n_k,n_\gamma) \\
0 \\
\end{pmatrix} .$$
\item For a complex end and complex $\alpha$,
$$\frac{\partial\mu_k}{\partial\log\alpha_\gamma} =\begin{pmatrix}
2\omega(n_k,n_\gamma) & 0 \\
0 & 2\omega(n_k,n_\gamma)\\
\end{pmatrix} .$$
\item For a quadrivalent vertex on the left, the $\beta$ coordinate is the imaginary part of the complex edge in the coordinate $\nu_W$, and we have
$$\frac{\partial\mu_k}{\partial\beta}=\left\{ \begin{array}{l}
2\omega(n_k,n_{\gamma_2})\frac{\beta-i}{\beta^2+1}\in\CC \text{ on the left (complex) side,} \\
2\omega(n_k,n_{\gamma_1})\frac{2\beta}{\beta^2+1} \in\RR \text{ on the right (real) side.} \\
\end{array}\right. $$
\item For a quadrivalent vertex on the right, the $\beta$ coordinate is the imaginary part of the complex edge in the coordinate $\nu_W$, and we get
$$\frac{\partial\mu_k}{\partial\beta}=\left\{ \begin{array}{l}
2\omega(n_k,n_{\gamma_2})\frac{\beta+i}{\beta^2+1} \in\CC\text{ on the right (complex) side,} \\
2\omega(n_k,n_{\gamma_1})\frac{2\beta}{\beta^2+1} \in\RR \text{ on the left (real) side.} \\
\end{array}\right.  $$
\item For a pentavalent vertex, let $(u,v)$ be the imaginary parts of the complex markings in coordinate $\nu_W$, so that their total coordinates are $iu$ and $1+iv$. On the left side,
\begin{align*}
\frac{\partial\mu_k}{\partial u}= \omega(n_k,n_{\gamma_2}) \frac{-2(u+i)}{(iu-1-iv)(iu-1+iv)} , \\
\frac{\partial\mu_k}{\partial v}= \omega(n_k,n_{\gamma_2}) \frac{2v}{(iu-1-iv)(iu-1+iv)}, \\
\end{align*}
and on the right side,
\begin{align*}
\frac{\partial\mu_k}{\partial u}= \omega(n_k,n_{\gamma_1}) \frac{2u}{(1+iv-iu)(1+iv+iu)}, \\
\frac{\partial\mu_k}{\partial v}= \omega(n_k,n_{\gamma_1}) \frac{2(-v+i)}{(1+iv-iu)(1+iv+iu)}.\\
\end{align*}
\end{itemize}

For a leaf $\lambda_k$, let $r(\lambda_k)$ denotes its rank among the leaves, with complex leaves counted twice since they contribute two coordinates. Thus, the rank among the coordinates is $r(\lambda_k)+1$ for a real leaf or the real part of a complex leaf, and $r(\lambda_k)+2$ for the imaginary part of a complex leaf. Let also $\hat{r}(\lambda_k)$ be the number of hatted coordinates before $\lambda_k$, equal to $1$ for $\lambda_k$ in the left part of $W_\mathrm{root}$ since we only forget $\re\nu_\mathrm{root}(\lambda_1)$, and $2$ on the right part since we also forget $\re\nu_\mathrm{root}(\lambda_{l_\mathrm{root}+1})$. Hence, the index of the row associated to a leaf $\lambda_k$ in the evaluation matrix is $\rkn-2+r(\lambda_k)+1(+1)$ for a real leaf (the $(+1)$ for the imaginary part), while the index of the corresponding column is $\rkn+r(\lambda_k)+1(+1)-\hat{r}(\lambda_k)$.

\medskip

In each of the following lemmas, we only write the columns of the Jacobian matrix corresponding to the coordinates associated to the pruned vertex $W$, and the rows corresponding to the leaves adjacent to $W$.

\begin{lem}
If $W_l$ is a trivalent real vertex adjacent to two real leaves, then one has
$$\det\ev(\Gamma)=(-1)^{1+\hat{r}(\lambda_{l+1})}\omega_{W_l}\det\ev(\widehat{\Gamma}),$$
where $\widehat{\Gamma}$ is the curve where $W_l$ has been removed and transformed into an end with the same direction.
\end{lem}

\begin{proof}
We consider a real trivalent vertex $W_l$ adjacent to two real ends of $\Gamma$. The only coordinate associated to $W_l$ is the $\alpha$ coordinate $\alpha_l$. With $\omega_W=2\omega(n_l,n_{l+1})$, the Jacobian matrix has then the following form:
\begin{center}
\begin{tabular}{cc}
                        & $\alpha$                       \\ \cline{2-2} 
\multicolumn{1}{c|}{$\mu_l$} & \multicolumn{1}{c|}{$\omega_W$}  \\ \cline{2-2} 
\multicolumn{1}{c|}{$\mu_{l+1}$} & \multicolumn{1}{c|}{$-\omega_W$} \\ \cline{2-2} 
\end{tabular}
\end{center}
Thus, we add the first row $l$ to the second $l+1$ and develop regarding the column. The column is in the place of the coordinate of $\nu_\mathrm{root}(\lambda_{l+1})$, ranked $\rkn+1+r(\lambda_{l+1})-\hat{r}(\lambda_{l+1})$ and the row is ranked $\rkn-1+r(\lambda_{l})$. As $r(\lambda_{l+1})=r(\lambda_l)+1$, we have 
$$\det\ev(\Gamma)=(-1)^{1+\hat{r}(\lambda_{l+1})}\omega_W\det\ev(\widehat{\Gamma}),$$
where $\hat{\Gamma}$ is the tropical curve where $W_l$ has been removed and replaced by an end with the same direction.
\end{proof}

\begin{lem}
If $W_l$ is a trivalent complex vertex, then we have
$$\det\ev(\Gamma)=\omega_{W_l}^2\det\ev(\widehat{\Gamma}),$$
where $\widehat{\Gamma}$ the curve where $W_l$ has been removed and transformed into an end with the same direction.
\end{lem}

\begin{proof}
We consider a complex trivalent vertex $W_l$ adjacent to two complex ends of $\Gamma$. The only coordinate associated to $W_l$ is the complex $\alpha$ coordinate $\alpha_l$. With $\omega_W=2\omega(n_l,n_{l+1})$, the Jacobian matrix has then the following form:
\begin{center}
\begin{tabular}{ccc}
                                         & \multicolumn{2}{c}{$\alpha$}            \\ \cline{2-3} 
\multicolumn{1}{c|}{\multirow{2}{*}{$\mu_l$}} & $\omega_W$  & \multicolumn{1}{c|}{$0$}      \\
\multicolumn{1}{c|}{}                    & $0$      & \multicolumn{1}{c|}{$\omega_W$}  \\ \cline{2-3} 
\multicolumn{1}{c|}{\multirow{2}{*}{$\mu_{l+1}$}} & $-\omega_W$ & \multicolumn{1}{c|}{$0$}      \\
\multicolumn{1}{c|}{}                    & $0$      & \multicolumn{1}{c|}{$-\omega_W$} \\ \cline{2-3} 
\end{tabular}
\end{center}
Thus, we add the pair of rows corresponding to $l$ and $l+1$ and develop regarding the pair of columns. Since we develop regarding a pair of columns, we do not need to care about the sign of the cofactor. Thus, we get
$$\det\ev(\Gamma)=\omega_W^2\det\ev(\widehat{\Gamma}),$$
where $\widehat{\Gamma}$ is the tropical curve where $W_l$ has been removed and replaced by a complex end with the same direction.
\end{proof}

\begin{lem}
If $W_l$ is a quadrivalent vertex, and $\beta$ is the coordinate associated to the vertex, then one has
$$\det\ev(\Gamma)=\omega_W^2\frac{1}{\beta^2+1}\det\ev(\widehat{\Gamma}),$$
where $\widehat{\Gamma}$ is the curve where $W_l$ has been replaced by a real end with the same direction.
\end{lem}

\begin{proof}
We consider a quadrivalent vertex $W_l$ adjacent to a complex end and a real end. The coordinates associated to $W_l$ are its $\alpha$ coordinate, and the coordinate $\beta$, equal to the imaginary part of the complex end. They are in the order $(\beta,\alpha)$ if the vertex is quadrivalent on the left, and $(\alpha,\beta)$ if quadrivalent on the right. First, assume that it is quadrivalent on the right. The Jacobian matrix is as follows:
\begin{center}
\begin{tabular}{ccc}
                                         & $\alpha$                   & $\beta$                       \\ \cline{2-3} 
\multicolumn{1}{c|}{$\mu_l$}                  & \multicolumn{1}{c|}{$\omega_W$} & \multicolumn{1}{c|}{$\omega_W\frac{2\beta}{\beta^2+1}$}  \\ \cline{2-3} 
\multicolumn{1}{c|}{\multirow{2}{*}{$\mu_{l+1}$}} & \multicolumn{1}{c|}{$-\omega_W$} & \multicolumn{1}{c|}{$-\omega_W\frac{\beta}{\beta^2+1}$} \\
\multicolumn{1}{c|}{}                    & \multicolumn{1}{c|}{$0$} & \multicolumn{1}{c|}{$-\omega_W\frac{1}{\beta^2+1}$}      \\ \cline{2-3} 
\end{tabular}
\end{center}
We add to the central row the first one and $\beta$ times the third. Then, we develop regarding the pair of columns and the first and third rows, to get
$$\det\ev(\Gamma)=\omega_W^2\frac{1}{\beta^2+1}\det\ev(\widehat{\Gamma}),$$
where $\widehat{\Gamma}$ is the curve where $W_l$ has been replaced by a real end.

If $W_l$ is quadrivalent on the left, the coordinates associated to $W_l$, the Jacobian matrix is this time
\begin{center}
\begin{tabular}{ccc}
                                         & $\beta$                       & $\alpha$                              \\ \cline{2-3} 
\multicolumn{1}{c|}{\multirow{2}{*}{$\mu_l$}} & \multicolumn{1}{c|}{$\omega_W\frac{\beta}{\beta^2+1}$}  & \multicolumn{1}{c|}{$\omega_W $}   \\
\multicolumn{1}{c|}{}                    & \multicolumn{1}{c|}{$\omega_W \frac{-1}{\beta^2+1}$}      & \multicolumn{1}{c|}{$0$}        \\ \cline{2-3} 
\multicolumn{1}{c|}{$\mu_{l+1}$}                  & \multicolumn{1}{c|}{$-\omega_W \frac{2\beta}{\beta^2+1}$} & \multicolumn{1}{c|}{$-\omega_W$} \\ \cline{2-3} 
\end{tabular}
\end{center}
and we proceed similarly to get the result: use the last two rows to cancel the first one and then develop to obtain once again
$$\det\ev(\Gamma)=\omega_W^2\frac{1}{\beta^2+1}\det\ev(\widehat{\Gamma}).$$
\end{proof}

\begin{lem}
If $W_l$ is a pentavalent vertex, up to positive scalar, one has
$$\det\ev(\Gamma)  = (-1)^{1+\hat{r}(\lambda_{l+1})}\det\ev(\widehat{\Gamma}).$$
\end{lem}

\begin{proof}
Let $W_l$ be a pentavalent vertex adjacent to two complex ends. The coordinates associated to $W_l$ are the $\alpha$ coordinate $\nu_L(\lambda_{l+1})-\nu_L(\lambda_{\kappa(l)})$ and the pair $(u,v)$ of coordinates which correspond to the imaginary parts of the complex markings. Those markings thus have coordinates $\nu_l$ equal to $iu$ and $1+iv$. Up to a division by $(1+v^2-u^2)^2+4u^2$ of the first and third columns, the Jacobian matrix is as follows:
\begin{center}
\begin{tabular}{cccc}
                                         & $u$                             & $\alpha$                       & $v$                             \\ \cline{2-4} 
\multicolumn{1}{c|}{\multirow{2}{*}{$\mu_l$}} & \multicolumn{1}{c|}{$\omega_W u(1+u^2-v^2)$} & \multicolumn{1}{c|}{$\omega_W$}  & \multicolumn{1}{c|}{$\omega_W v(1+v^2-u^2)$} \\
\multicolumn{1}{c|}{}                    & \multicolumn{1}{c|}{$\omega_W(-1-u^2-v^2)$}   & \multicolumn{1}{c|}{$0$}      & \multicolumn{1}{c|}{$\omega_W 2uv$}      \\ \cline{2-4} 
\multicolumn{1}{c|}{\multirow{2}{*}{$\mu_{l+1}$}} & \multicolumn{1}{c|}{$-\omega_W u(1+u^2-v^2)$} & \multicolumn{1}{c|}{$-\omega_W$} & \multicolumn{1}{c|}{$-\omega_W v(1+v^2-u^2)$} \\
\multicolumn{1}{c|}{}                    & \multicolumn{1}{c|}{$-\omega_W (-2uv)$}      & \multicolumn{1}{c|}{$0$}      & \multicolumn{1}{c|}{$-\omega_W(1+u^2+v^2)$}    \\ \cline{2-4} 
\end{tabular}
\end{center}
We use the middle column to cancel the first and third columns first and third rows. We then develop with respect to the considered columns. This multiplies the determinant by
$$\det \begin{pmatrix}
\omega_W(-1-u^2-v^2) & \omega_W 2uv \\
-\omega_W (-2uv) & -\omega_W(1+u^2+v^2) \\
\end{pmatrix}=\omega_W^2(1+(v-u)^2)(1+(u+v)^2)>0.$$
Then, we are left with a unique column having coefficients $\omega_W$ and $-\omega_W$, and we proceed as in the real trivalent case, getting up to positive scalar
$$\det\ev(\Gamma)  = (-1)^{1+\hat{r}(\lambda_{l+1})}\det\ev(\widehat{\Gamma}).$$
\end{proof}

\subsection{The hypersurface of fixed momenta}

Let $\Delta_B=(n_k)\subset N$ be a degree, which is meant to be the degree of the component $C_B$ of a reducible curve $C_A\cup C_B$. To adapt from the previous notations, the cardinal of $\Delta_B$ is then denoted by $q$, and the coordinates of the marked points on the curve will be denoted with a letter $b$. To keep definitions simple, we assume that all the marked points are real, but the statements immediately translate to the case of complex marked points. We consider the marked points to be $(p_0,p_1,\dots,p_q)$, where $p_k$ for $k>0$ are sent to the toric boundary with cocharacter $n_k$, and $p_0$ is an additional marked point mapped to the main strata of the toric variety. In particular, we have the forgetful map
$$\mathrm{ft}:\MM_{0,q+1}^\tau\longrightarrow\MM_{0,q}^\tau,$$
that forgets the position of the marked point $p_0$. This forgetful map can be seen as the universal curve for the moduli space $\MM_{0,q}^\tau$. For a component $\MM(\beta)$ of $\MM_{0,q}^\tau$, $\mathrm{ft}^{-1}\left(\MM(\beta)\right)\subset\MM_{0,q+1}^\tau$ consists in the curves where the cyclic order on the marked points $p_1,\dots,p_q$ is fixed. We can write
$$\mathrm{ft}^{-1}\MM(\beta)=\bigcup_{s=0}^{q-1}\MM(\beta_s),$$
where $\MM(\beta_s)$ is the component of $\MM_{0,q+1}^\tau$ corresponding to the order
$$p_0<\beta(s+1)<\cdots<\beta(q)<\beta(1)<\cdots<\beta(s)<p_0.$$
Moreover, as the fibers of the forgetful map have a natural orientation, an orientation on $\MM(\beta)$ induces an orientation on $\mathrm{ft}^{-1}\left(\MM(\beta)\right)$. We now give three descriptions of the same hypersurface $\mathcal{H}(\mu)$, called \textit{hypersurface of fixed momenta}.

\begin{itemize}[label=$\circ$]
\item First, we use the moment map:
$$\mom:N_\RR\times\MM_{0,q+1}^\tau\longrightarrow \RR^{q-1},$$
that evaluates the moment at the marked points $p_1,\dots,p_{q-1}$. Since we have considered a free additional marked point, for a fixed value $\mu\in\RR^{q-1}$, the preimage, consisting of curves with fixed moments, is a subvariety of dimension $\rkn-1$ inside $N_\RR\times\MM_{0,q+1}^\tau$. Thus, the evaluation at the marked point $p_0$ gives a hypersurface $\mathcal{H}(\mu)=\ev_0(\mom^{-1}(\mu))$ in $N_\RR$. Choosing a coordinate such that $p_0$ is $\infty$, the evaluation $\ev_0$ at $p_0$ is just the projection on the $N_\RR$ factor. If we restrict to $\mathrm{ft}^{-1}\left(\MM(\beta)\right)$, we denote the hypersurface by $\mathcal{H}^\beta(\mu)$.

\item Alternatively, if one forgets the marked point $p_0$ and considers the moment map
$$\mom':N_\RR\times\MM_{0,q}^\tau\longrightarrow\RR^{q-1},$$
the preimage $\mom'^{-1}(\mu)$ is a subvariety of dimension $\rkn-2$, and $\mathcal{H}(\mu)$ can be seen as the union of the loci of these curves. In particular, $\mathcal{H}(\mu)$ is thus foliated by these curves. Similarly, we define $\mathcal{H}^\beta(\mu)$ restricting to $\MM(\beta)$

\item Last, we consider the map
$$\ev_0\times\mom:N_\RR\times\MM_{0,q+1}^\tau\longrightarrow N_\RR\times\RR^{q-1}.$$
Its image is a hypersurface $\widehat{\mathcal{H}}$. We then have $\mathcal{H}(\mu)=\pi_1(\widehat{\mathcal{H}}\cap N_\RR\times\{\mu\})$, where $\pi_1$ denotes the projection on the first factor. Similarly, we recover $\mathcal{H}^\beta(\mu)$ by projecting after intersecting with $\widehat{\mathcal{H}}^\beta=(\ev_0\times\mom)(N_\RR\times\mathrm{ft}^{-1}\left(\MM(\beta)\right))$.
\end{itemize}

Each $\widehat{\mathcal{H}}^\beta(\mu)$ is further subdivided according to the position of the marked point $p_0$, we obtain for each $s$,
$$\mathcal{H}^{\beta_s}(\mu)=\ev_0(\mom^{-1}(\mu)\cap\MM(\beta_s)).$$

Notice that a choice of orientation on $\MM(\beta_s)$ and $\RR^{q-1}$ induces an orientation on $\mom^{-1}(\mu)\cap\MM(\beta_s)=\mom|_{\MM(\beta_s)}^{-1}(\mu)$, and thus an orientation on $\mathcal{H}^{\beta_s}(\mu)$.

\medskip

Assuming $\omega$ is non-degenerate, we have the following statement.

\begin{prop}
On each $\mathcal{H}^\beta(\mu)\subset N_\RR$, the foliation provided by the family of curves coincide with the characteristic foliation provided by the symplectic form $\omega$ on $N_\RR$.
\end{prop}

\begin{proof}
Consider a family 
$$f_t:y\longmapsto \xi(t)+\sum_{j=1}^q 2n_j \log(y-b_j(t)),$$
of parametrized oriented real curves with fixed moments $\mu$, where the coordinate $y$ is chosen so that $p_0$ is $\infty$. In particular, each function
$$\omega(n_k,\xi(t))+\sum_{j=1}^q 2\omega(n_k,n_j)\log(b_k(t)-b_j(t))=\mu_k,$$
is constant. Thus, for each $k$, one has the relation
$$\omega(n_k,\xi'(t))+\sum_{j=1}^q 2\omega(n_k,n_j)\frac{b_k'(t)-b_j'(t)}{b_k(t)-b_j(t)}=0.$$
Using the coordinate $z=-\frac{1}{y}$, for which the parametrization becomes
$$\hat{f}_t:z\longmapsto\xi(t)+\sum_{j=1}^q 2n_j\log(1+b_jz),$$
we compute the tangent vector $\tau_0$ to the curve at the point of parameter $y=\infty$, \textit{i.e.} $z=0$ which is mapped to $\xi(t)\in N_\RR$:
$$\tau_0=\left. \frac{\partial \hat{f}_t}{\partial z}\right|_{z=0}=\sum_{k=1}^q 2b_k n_k.$$
Meanwhile, the evaluation at $p_0$ being the projection $N_\RR\times\MM(\beta)\rightarrow N_\RR$, the derivative of $\mathrm{ev}_0(f_t)$ is just $\xi'(t)$. Hence, to prove that the two foliations coincide, we just need to prove that
$$\omega\left(\sum_{k=1}^q 2b_k(t)n_k,\xi'(t)\right) =0,$$
so that the tangent space to $\mathcal{H}^\beta(\mu)$ is $\omega$-orthogonal to $\tau_0$. Using the previously mentioned relations and the antisymmetry of $\omega$, we get
\begin{align*}
\omega\left(\sum_{k=1}^q 2b_k(t)n_k,\xi'(t)\right) & = 2\sum_{k=1}^q b_k(t)\omega(n_k,\xi'(t)) \\
& = -2\sum_{k=1}^q \sum_{j=1}^q 2\omega(n_k,n_j)\frac{b_k'-b_j'}{b_k-b_j} \\
& = -4\left( \sum_{j,k}\frac{\omega(n_k,n_j)}{b_k-b_j}b_k' - \sum_{j,k} \frac{\omega(n_k,n_j)}{b_k-b_j}b_j' \right) =0,\\
\end{align*}
yielding the result.
\end{proof}

If $\omega$ is degenerate, it suffices to translate the result by the action of the kernel of $\omega$.

\begin{rem}
In particular, as the foliating curves are oriented, this fixes an orientation of $\mathcal{H}^\beta(\mu)$ by choosing the normal vector to be $\iota_{\tau_0}\omega$, where $\tau_0$ is the oriented tangent vector to the leaf of the foliation passing through the point. Moreover, one can check that this orientation coincides with the previous orientations described on $\mathcal{H}^\beta(\mu)$. Therefore, to compute the intersection index with a curve $C_A$ inside $N_\RR$, one just needs to evaluate $\iota_{\tau_0}\omega$ on the tangent vector to the curve $C_A$.
\end{rem}

\subsection{Compactification of the space of parametrized curves}
\label{generic choice omega}

%limit of curves with components on the boundary, and that it cannot happen with H

The moduli space $\MM_{0,R,S}^\tau$ is not compact, but can be compactified into the space of oriented real rational stable curves $\overline{\MM}_{0,R,S}^\tau$ which is compact. Similarly, the moduli space of parametrized curves $N_\RR\times\MM_{0,R,S}^\tau$ is not compact but can be compactified considering maps from oriented real rational stable curves to the toric variety $\CC\Delta$. Concretely, it means that some components of a reducible stable curve might be sent to the toric boundary of $\CC\Delta$.

\medskip

For a family of curves $C^{(t)}\in N_\RR\times\MM_{0,R,S}^\tau$ converging to $f:C^{(0)}\rightarrow\CC\Delta$, with $C^{(0)}$ a semi-stable curve, the disposition of the components of $C^{(0)}$ can be seen with the tropical limit: the vertices of the tropical limit curve correspond to the components of $C^{(0)}$. And the strata of $\CC\Delta$ that they are mapped to depends on the position of its image in the fan of $\CC\Delta$ inside $N_\RR$. In particular, the curve as a component in the dense orbit if and only if the tropical curve passes through $0$.

\medskip

Inside the moduli space $\Gamma\in N_\RR\times\MM_{0,N}^\tau$ of parametrized tropical curves of degree $\Delta_\mathrm{trop}$, the set of parametrized tropical curves passing through $0$ is a subfan of dimension $N-1$. Hence, if $\omega$ is chosen generically, we can assume that the moment map is injective on its cones, unless the curve has some \textit{flat vertex}: a vertex for which the outgoing edges belong to the same line. It implies that the only parametrized tropical curve passing through $0$ with moments $0$ is the curve with a unique vertex centered at $0$. In particular, the only curves having moment $0$ are reducible curves all of whose components are mapped to the main strata of the toric variety. This assumption further implies that $\omega$ does not vanish on $a_W\wedge b_W$ for any trivalent vertex of any rational tropical curve of degree $\Delta$, which is the assumption needed in \cite{blomme2020refined} to get tropical refined invariants in $\ZZ[\Lambda^2 N]$.

\begin{prop}
\label{prop reducible limit}
If $\omega$ is chosen in a generic way, a family of curves with fixed momenta cannot converge to a reducible curve, some of whose components are mapped to the toric boundary, and at least one is mapped to the main strata.
\end{prop}

\begin{proof}
Due to the generic choice, the only tropical curves with moments $0$ and passing through $0$ have a flat vertex. However, the parametrized curve associated to a flat vertex can factor as a map
$$\CC P^1\rightarrow\CC P^1\dashrightarrow N_{\CC^*},$$
where the second map is the cocharacter $t\in\CC^*\mapsto t^n\in N_{\CC^*}$ with $n$ the primitive vector directing the line containing the edges adjacent to the flat vertex, and the first map is some ramified cover of $\CC P^1$, with only $0,\infty\in\CC P^1$ being non-simple ramification points. The ramification profile of $0$ and $\infty$ is given by the slopes of the edges adjacent to the flat vertex. Such a curve cannot have distinct intersection points with the toric boundary, and thus cannot happen as a limit either. Therefore, there are no flat vertex, and the parametrized tropical curve is the trivial one.
\end{proof}

\begin{coro}
For a generic $\mu$, the hypersurfaces $\mathcal{H}^{\beta_s}(\mu)$ and $\mathcal{H}^{\beta}(\mu)$ are closed inside $N_\RR$.
\end{coro}

\begin{proof}
If $\mu$ is chosen generically, no subfamily of $\mu$ satisfies the Menelaus condition. Thus, there are no reducible curves with all components in the main strata in the adherence of $\mathcal{H}^\beta(\mu)$. Thus, the only possible degenerations are when some component is sent to the boundary. By Proposition \ref{prop reducible limit}, a reducible curve limit of curves in $\mathcal{H}^\beta(\mu)$ has then all its components sent to the boundary of the toric variety. Then, the boundary is sent to infinity and $\mathcal{H}^{\beta_s}(\mu)$ and $\mathcal{H}^{\beta}(\mu)$ are closed inside $N_\RR$.
\end{proof}

\begin{rem}
In the planar case considered in \cite{mikhalkin2017quantum}, it is another argument that allows one to avoid curves with components sent to the boundary: such curves would have too many intersection points with the toric boundary.
\end{rem}

\section{Proof of refined invariance}

We now get to the proof of Theorem \ref{theorem invariance}.

\begin{proof}[Proof of Theorem \ref{theorem invariance}] We want to prove that the refined count $R_{\Delta,\omega}(P,\mu)$ does not depend on the choice of $(P,\mu)$ as long as it is generic. Let $(P(t),\mu(t))_{t\in[0;1]}$ be a generic path between two regular values of the evaluation map. In particular, it is transverse to the evaluation map and $\ev^{-1}\{(P(t),\mu(t))\}$ is a manifold of dimension $1$.

For every $t$ outside a finite set $F$, $(P(t),\mu(t))$ is a regular value of the evaluation map. Notice that the quantum index is locally constant:
\begin{itemize}[label=$\circ$]
\item In the totally real case, this is immediate since the quantum index does only depend on the cyclic order on the marked points.
\item If there are complex constraints, the quantum index is a shift of the log-area, which is continuous. The shift depends on the arguments of the intersection points with the toric boundary, which are fixed by the chosen complex constraints. Thus, although the quantum index may take several values on the same component $\MM(\gamma,J)$, it is constant for a family of curves with bounded constraints.
\end{itemize}
As the sign  $\sigma$ used to count solutions is also locally constant, the refined count $R_{\Delta,\omega}(P(t),\mu(t))$ is locally constant, and therefore constant outside this finite set of values $F$. We now need to show the invariance of this count at a value $t_*\in F$.

\medskip

According to Proposition \ref{prop reducible limit}, when $t$ goes to $t_*$, a curve $C_t\in\ev^{-1}(P(t),\mu(t))$ can either go to some curve $C_{t_*}$ inside a component of $N_\RR\times\MM_{0,R,S}^\tau$, or to some reducible curve all of whose components are mapped to the main strata of the toric variety, thanks to the genericity of the choice of $\omega$. We need to show the invariance at both kinds of walls.
\begin{itemize}[label=$\bullet$]
\item In the first case, the invariance follows from the fact that the quantum index is locally constant on the components of $N_\RR\times\MM_{0,R,S}^\tau$ if we fix the constraints, and that the sign used to count the curves is the sign of the differential of the evaluation map. The local invariance results from the standard definition of the degree of a map between two manifolds.
\item In the second case, as the path $(P(t),\mu(t))$ is chosen generically, we can assume that $C_t$ goes to some reducible curve $C_A\cup C_B$, associated to a decomposition $\Delta_A\sqcup\Delta_B$ of $\Delta$. Let $k_A$ and $k_B$ be the quantum indices of $A$ and $B$ respectively, so that for $t$ close to $t_*$, one has $k(C_t)=k_A+k_B$. When $t$ passes through the value $t_*$, the node of $C_A\cup C_B$ gets resolved in a different way: if we keep the orientation on $C_A$ fixed, the quantum index passes from $k_A+k_B$ to $k_A-k_B$. Thus, we do not have local invariance on the domain. Fortunately, when a reducible curve is mapped to $(P(t),\mu(t))$ by the evaluation map, there are many other reducible curves that do so, and this provide invariance, as shown in the rest of the proof.
\end{itemize}

The rest of the proof is devoted to prove the invariance near a wall of reducible curves. We consider $C_t\in N_\RR\times\MM(\gamma,J)$ that degenerates to $C_A\cup C_B$, and denote the order on the points as in section \ref{section coordinates reducible curves}:
$$p_0<\alpha(1)<\cdots<\alpha(r)<\beta(1)<\cdots<\beta(q)<\alpha(r+1)\cdots<\alpha(p)<p_0.$$
Notice that varying the value of $r$ and the position of the node on the component $C_B$, the quantum index of the smoothing in the orientation preserving way has quantum index $k_A+k_B$. We need to find those solutions and prove that their signed count remains invariant.

\begin{itemize}[label=$\circ$]
\item Solving for the reducible curves $C_A\cup C_B$ mapped to $(P(t_*),\mu(t_*))$ by the evaluation map can be done as follows: letting $\mu(t_*)=(\mu_A,\mu_B)$, first, solve for $C_A$, which has fixed moments $\mu_A$ and maps $p_0$ to $P(t_*)$, and then solve for $C_B$, which has fixed moments $\mu_B$ and has to intersect $C_A$. For each possible $C_A$, the set of possible $C_B$ is in bijection with the intersection points of $C_A$ and the hypersurface of momenta $\mathcal{H}^\beta(\mu_B)$: the reducible curve corresponding to an intersection point is the union of $C_A$ and the leaf of the foliation of $\mathcal{H}^\beta(\mu_B)$ passing through the intersection point. We can further impose the condition that the intersection point between $C_A$ and $C_B$ belongs to a preferred component of the curves of $\MM(\beta)$, intersection with $\MM(\beta_s)$ for some $s$. Intersecting with $\mathcal{H}^{\beta_0}(\mu_B)$, we get the curves $C_B$ that meet $C_A$ and such that the node of $C_A\cup C_B$ belongs to the interval $[\beta(q);\beta(1)]$.
\item Next, we observe that as $C_A$ is contractible in $N_\RR$ and $\mathcal{H}^{\beta_0}(\mu_B)$ has no boundary, the signed number of intersection points between $C_A$ and $\mathcal{H}^{\beta_0}(\mu_B)$ is $0$. This is also true for the $\mathcal{H}^{\beta_s}(\mu_B)$. Moreover, choosing the orientation of $\mathcal{H}^{\beta_0}(\mu_B)$ given by the oriented foliation, this intersection number is
$$\sum_{C_B}\mathrm{sgn}\ \omega(\tau_A,\tau_B)=0,$$
where the sum is over the leaves of the foliation that intersect $C_A$. It means that there are the same number of leafs intersecting $C_A$ positively and negatively. According to Proposition \ref{prop which side of the wall}, this sign $\mathrm{sgn}\ \omega(\tau_A,\tau_B)$ also determines on which side of the wall the smoothing in the orientation preserving way lies. Thus, to conclude, we only need to check that all the smoothings are counted with the same sign $\sigma$, so that we have invariance.

\item We start with the totally real case whose situation is easier to handle. The sign $\sigma(C_t)$ is equal to the sign of the Jacobian of the evaluation map when the domain is endowed with the orientation $\oo(\nu_{1,N})$ and the rows corresponding to the codomain are ordered in the order prescribed by the orientation. We now choose on the domain the coordinates $(\widehat{a_1},\dots,a_r,\eta,b_2,\dots,b_{q-1},\delta,a_{r+1},\dots,\widehat{a_p})$, which induce the orientation $\oo(\nu_{1,N})$ if $r\neq 0,p$, and $\oo(\nu_{q+1,N})$, $\oo(\nu_{1,p})$ respectively else.

\medskip

First, we reorder the rows and columns to match the order
$$\alpha(1)<\cdots<\alpha(p)<\beta(1)<\cdots<\beta(q).$$
This changes the determinant by a sign $(-1)^q$ if $r\neq 0,p$. If $r=p$, there is nothing to move, but $\oo(\nu_{1,p})=(-1)^q\oo(\nu_{1,N})$, so we still get the $(-1)^q$. If $r=0$, doing the reordering does not affect the sign, but $\oo(\nu_{q+1,N})=(-1)^q\oo(\nu_{1,N})$, so we still get the $(-1)^q$.

\medskip

We now compute the Jacobian matrix in the above coordinate. The moments at the marked points take the following form:
\begin{align*}
\mu_{\alpha(k)} = &  \omega(n_{\alpha(k)},\xi) 
	 + \sum_{i=1}^p 2\omega(n_{\alpha(k)},n_{\alpha(i)}) \log|a_k-a_i| 
	 + \sum_{j=1}^q 2\omega(n_{\alpha(k)},n_{\beta(j)}) \log|a_k-(\eta+\delta b_j)| \\
\mu_{\beta(k)} = &  \omega(n_{\beta(k)},\xi)
	 + \sum_{i=1}^p 2\omega(n_{\beta(k)},n_{\alpha(i)}) \log|\eta+\delta b_k-a_i| 
	 + \sum_{j=1}^q 2\omega(n_{\beta(k)},n_{\beta(j)}) \log|b_k-b_j| \\
\end{align*}

An elementary computation yields the partial derivatives, which have the following limits when $\delta\rightarrow 0$:
\begin{align*}
\frac{\partial\mu_{\alpha(k)}}{\partial a_l}  & \xrightarrow[\delta\to 0]{}
\left\{ \begin{array}{l}
-\frac{ 2\omega(n_{\alpha(k)},n_{\beta(l)}) }{a_k-a_l} \text{ if }l\neq k \\
\sum_{i=1}^p \frac{ 2\omega(n_{\alpha(k)},n_{\alpha(i)}) }{a_k-a_i}\text{ if }l=k\\
\end{array} \right. ,
&  \frac{\partial\mu_{\alpha(k)}}{\partial b_l}  & \xrightarrow[\delta\to 0]{} 0,\\
\frac{\partial\mu_{\beta(k)}}{\partial a_l}  & \xrightarrow[\delta\to 0]{} -\frac{ 2\omega(n_{\beta(k)},n_{\alpha(l)}) }{\eta-a_l},
& \frac{\partial\mu_{\beta(k)}}{\partial b_l} & \xrightarrow[\delta\rightarrow 0]{} \left\{ \begin{array}{l}
-\frac{ 2\omega(n_{\beta(k)},n_{\beta(l)}) }{b_k-b_l} \text{ if }l\neq k \\
\sum_{j=1}^q \frac{ 2\omega(n_{\beta(k)},n_{\beta(j)}) }{b_k-b_j} \text{ if }l=k\\
\end{array} \right. ,  \\
\end{align*}

\begin{align*}
\frac{\partial\mu_{\alpha(k)}}{\partial \eta}  \xrightarrow[\delta\to 0]{} & 
 0, &
\frac{\partial\mu_{\alpha(k)}}{\partial \delta} \xrightarrow[\delta\to 0]{} & 
 -\frac{1}{a_k-\eta}\sum_{j=1}^q 2\omega(n_{\alpha(k)},n_{\beta(j)})b_j \\
 & & & =-\frac{1}{a_k-\eta}\omega(n_{\alpha(k)},\tau_B),\\
\frac{\partial\mu_{\beta(k)}}{\partial \eta}  \xrightarrow[\delta\to 0]{} & 
 \sum_{i=1}^p \frac{2\omega(n_{\beta(k)},n_{\alpha(i)})}{\eta - a_i}   , &
\frac{\partial\mu_{\beta(k)}}{\partial \delta} \xrightarrow[\delta\to 0]{}  & 
\sum_{i=1}^p 2\omega(n_{\beta(k)},n_{\alpha(i)})\frac{b_k}{\eta - a_i}. \\
 & =\omega(n_{\beta(k)},\tau_A) & & \\
\end{align*}
The Jacobian matrix is as follows:

\begin{center}
\begin{tabular}{ccllcccccccc}
                                           & \multicolumn{3}{c}{$\xi$}                  & $\widehat{a_1}$      & $\cdots$  & $\widehat{a_p}$ & $\eta$                                           & $b_2$      & $\cdots$  & $b_{q-1}$                          & $\delta$                                         \\ \cline{2-12} 
\multicolumn{1}{c|}{\multirow{3}{*}{$\ev_{m_i}$}} & \multicolumn{3}{c|}{\multirow{3}{*}{*}} & 0      & $\cdots$  & \multicolumn{1}{c|}{0}      & \multicolumn{1}{c|}{0}                        & 0      & $\cdots$  & \multicolumn{1}{c|}{0}     & \multicolumn{1}{c|}{0}                        \\
\multicolumn{1}{c|}{}                      & \multicolumn{3}{c|}{}                   & $\vdots$  & $\ddots$  & \multicolumn{1}{c|}{$\vdots$}  & \multicolumn{1}{c|}{$\vdots$}                    & $\vdots$  & $\ddots$  & \multicolumn{1}{c|}{$\vdots$} & \multicolumn{1}{c|}{$\vdots$}                    \\
\multicolumn{1}{c|}{}                      & \multicolumn{3}{c|}{}                   & 0      & $\cdots$  & \multicolumn{1}{c|}{0}      & \multicolumn{1}{c|}{0}                        & 0      & $\cdots$  & \multicolumn{1}{c|}{0}     & \multicolumn{1}{c|}{0}                        \\ \cline{2-12} 
\multicolumn{1}{c|}{\multirow{3}{*}{$\mu_\alpha$}}  & \multicolumn{3}{c|}{$\iota_n\omega$}            & \multicolumn{3}{c|}{\multirow{3}{*}{$\partial_a\mu_\alpha$}} & \multicolumn{1}{c|}{0}                        & \multicolumn{3}{c|}{\multirow{3}{*}{0}}       & \multicolumn{1}{c|}{\multirow{3}{*}{$\partial_\delta\mu_\alpha$}} \\
\multicolumn{1}{c|}{}                      & \multicolumn{3}{c|}{$\vdots$}              & \multicolumn{3}{c|}{}                         & \multicolumn{1}{c|}{$\vdots$}                    & \multicolumn{3}{c|}{}                         & \multicolumn{1}{c|}{}                         \\
\multicolumn{1}{c|}{}                      & \multicolumn{3}{c|}{$\iota_n\omega$}            & \multicolumn{3}{c|}{}                         & \multicolumn{1}{c|}{0}                        & \multicolumn{3}{c|}{}                         & \multicolumn{1}{c|}{}                         \\ \cline{2-12} 
\multicolumn{1}{c|}{\multirow{3}{*}{$\mu_\beta$}}  & \multicolumn{3}{c|}{$\iota_n\omega$}            & \multicolumn{3}{c|}{\multirow{3}{*}{$\partial_a\mu_\beta$}} & \multicolumn{1}{c|}{\multirow{3}{*}{$\partial_\eta\mu_\beta$}} & \multicolumn{3}{c|}{\multirow{3}{*}{$\partial_b\mu_\beta$}} & \multicolumn{1}{|c|}{\multirow{3}{*}{$\partial_\delta\mu_\beta$}} \\
\multicolumn{1}{c|}{}                      & \multicolumn{3}{c|}{$\vdots$}              & \multicolumn{3}{c|}{}                         & \multicolumn{1}{c|}{}                         & \multicolumn{3}{c|}{}                         & \multicolumn{1}{c|}{}                         \\
\multicolumn{1}{c|}{}                      & \multicolumn{3}{c|}{$\iota_n\omega$}            & \multicolumn{3}{c|}{}                         & \multicolumn{1}{c|}{}                         & \multicolumn{3}{c|}{}                         & \multicolumn{1}{c|}{}                         \\ \cline{2-12} 
\end{tabular}
\end{center}

\medskip

We notice that the matrix comprised of the six top left blocks is precisely the Jacobian square matrix for the evaluation map with respect to $C_A$, and does not depend on $C_B$. Unfortunately, we cannot develop right away because of the last column. However, we notice that due to $\sum n_{\beta(k)}=0$ the sum of the rows corresponding to $\mu_\beta$ is $0$ except in the last column. In the last column, we get $-\frac{1}{2}\omega(\tau_A,\tau_B)$. Then, we can develop with respect to the last row to get a block triangular matrix, and get rid of the first block that does not depend on $C_B$. We are left with the following matrix:
\begin{center}
\begin{tabular}{ccccc}
                           & $\eta$                                           & $b_2$            & $\cdots$            & $b_{q-1}$           \\ \cline{2-5} 
\multicolumn{1}{c|}{$\mu_{\beta(1)}$}   & \multicolumn{1}{c|}{\multirow{3}{*}{$\omega(n_{\beta(k)},\tau_A)$}} & \multicolumn{3}{c|}{\multirow{3}{*}{$\partial_b\mu_\beta$}} \\
\multicolumn{1}{c|}{$\vdots$} & \multicolumn{1}{c|}{}                         & \multicolumn{3}{c|}{}                         \\
\multicolumn{1}{c|}{$\mu_{\beta(q-1)}$}   & \multicolumn{1}{c|}{}                         & \multicolumn{3}{c|}{}                         \\ \cline{2-5} 
\end{tabular}
\end{center}

\medskip

To show that the total sign does not depend on the chosen $C_B$ intersecting $C_A$, we need to prove that up to a global sign, the sign of this matrix is once again $\omega(\tau_A,\tau_B)$, \textit{i.e.} the intersection index between $C_A$ and $\mathcal{H}^{\beta_0}(\mu_B)$ at the corresponding point. This is the case since both correspond in fact to the intersection index at $C_B$ between $C_A\times\{\mu_B\}$ and $\widehat{\mathcal{H}}^{\beta_0}$ inside $N_\RR\times \RR^{q-1}$. The two ways to compute the intersection index are as follows.
\begin{itemize}[label=$\triangleright$]
\item First, recall that if $\MM(\beta_0)$ denotes the component of $\MM_{0,q+1}^\tau$ corresponding to the order
$$p_0<\beta(1)<\cdots<\beta(q)<p_0,$$
then $\widehat{\mathcal{H}}^{\beta_0}$ is the image of
$$\ev_0\times\mom:N_\RR\times\MM(\beta_0)\longrightarrow N_\RR\times \RR^{q-1}.$$
As $\MM(\beta_0)$ is oriented, \textit{e.g.} with $\oo(\nu_{1,q})$, the image is naturally oriented. We get the following matrix:

\begin{center}
\begin{tabular}{cccccccc}
                           &                                           & \multicolumn{3}{c}{$\xi$}               & $b_2$ & $\cdots$   & $b_{q-1}$                     \\ \cline{2-8} 
\multicolumn{1}{c|}{}      & \multicolumn{1}{c|}{\multirow{3}{*}{$\tau_A$}} & $1$ &         & \multicolumn{1}{c|}{$0$} &   &         & \multicolumn{1}{c|}{} \\
\multicolumn{1}{c|}{$\ev_0$}    & \multicolumn{1}{c|}{}                     &   & $\ddots$   & \multicolumn{1}{c|}{}  &   & $0$       & \multicolumn{1}{c|}{} \\
\multicolumn{1}{c|}{}      & \multicolumn{1}{c|}{}                     & $0$ &         & \multicolumn{1}{c|}{$1$} &   &         & \multicolumn{1}{c|}{} \\ \cline{2-8} 
\multicolumn{1}{c|}{$\mu_{\beta(1)}$}   & \multicolumn{1}{c|}{$0$}                    &   &         & \multicolumn{1}{c|}{}  &   &         & \multicolumn{1}{c|}{} \\
\multicolumn{1}{c|}{$\vdots$} & \multicolumn{1}{c|}{$\vdots$}                &   & $\iota_n\omega$ & \multicolumn{1}{c|}{}  &   & $\partial_b\mu_\beta$ & \multicolumn{1}{c|}{} \\
\multicolumn{1}{c|}{$\mu_{\beta(q-1)}$}   & \multicolumn{1}{c|}{0}                    &   &         & \multicolumn{1}{c|}{}  &   &         & \multicolumn{1}{c|}{} \\ \cline{2-8} 
\end{tabular}
\end{center}

\medskip

We then use the columns corresponding to $\xi$ to cancel the first rows of the first column, and then develop with respect to the identity block. We are left with the desired matrix for $C_B\in\MM(\beta_0)$, which thus computes the intersection index between $C_A\times\{\mu_B\}$ and $\widehat{\mathcal{H}}^{\beta_0}$.

\item The second way of computing the intersection index is easier: $(C_A\times\{\mu_B\})\cap\widehat{\mathcal{H}}^{\beta_0}$ is computed by taking $\tau_A$, and a basis of the tangent space to $\widehat{\mathcal{H}}^\beta$ that is constructed as follows: take a basis of $\widehat{\mathcal{H}}^{\beta_0}\cap(N_\RR\times\{\mu_B\})=\mathcal{H}^{\beta_0}(\mu_B)\times\{\mu_B\}$ which is oriented using the foliation by oriented curves and the $2$-form, and concatenate it with a basis that projects to an oriented basis of $H$. The intersection index is then equal to $\omega(\tau_A,\tau_B)$, since we only need to evaluate $\iota_{\tau_B}\omega$ on $\tau_A$.
\end{itemize}

\item If there are complex setting, we proceed similarly with the corresponding coordinates described in section \ref{section coordinates reducible curves}. The computation is more technical because of the two types of coordinates on the moduli space.

\end{itemize}
\end{proof}

\begin{rem}
The key part of the proof is the one that asserts that the total signed intersection between $\mathcal{H}^\beta(\mu_B)$ and each $C_A$ is $0$. The corresponding part in the proof of the planar case in \cite{mikhalkin2017quantum} seems different because in that case, the hypersurface $\mathcal{H}^\beta(\mu_B)$ consists of a finite number of curves $C_B$, and each $C_A$ has several intersection points with each of these components. In higher dimensions, the various intersection points between $C_A$ and $\mathcal{H}^\beta(\mu_B)$ might not belong to the same leaf of the foliation.
\end{rem}

\bibliographystyle{plain}
\bibliography{biblio}

\begin{thebibliography}{10}

\bibitem{block2016refined}
Florian Block and Lothar G{\"o}ttsche.
\newblock Refined curve counting with tropical geometry.
\newblock {\em Compositio Mathematica}, 152(1):115--151, 2016.

\bibitem{blomme2020phd}
Thomas Blomme.
\newblock Computation of refined enumerative invariants in real and tropical
  geometry.
\newblock {\em PhD Thesis}, 2020.

\bibitem{blomme2020computation}
Thomas Blomme.
\newblock Computation of refined toric invariants ii.
\newblock {\em arXiv preprint arXiv:2007.02275}, 2020.

\bibitem{blomme2020refined}
Thomas Blomme.
\newblock Refined count for rational tropical curves in arbitrary dimension.
\newblock {\em arXiv preprint arXiv:2010.05777}, 2020.

\bibitem{blomme2020tropical}
Thomas Blomme.
\newblock A tropical computation of refined toric invariants.
\newblock {\em arXiv preprint arXiv:2001.09305}, 2020.

\bibitem{bousseau2020mirror}
Pierrick Bousseau.
\newblock Quantum mirrors of log {C}alabi--{Y}au surfaces and higher-genus
  curve counting.
\newblock {\em Compositio Mathematica}, 156(2):360--411, 2020.

\bibitem{bousseau2020vertex}
Pierrick Bousseau.
\newblock The quantum tropical vertex.
\newblock {\em Geometry \& Topology}, 24(3):1297--1379, 2020.

\bibitem{filippini2012block}
Sara~Angela Filippini and Jacopo Stoppa.
\newblock Block-g$\backslash$" ottsche invariants from wall-crossing.
\newblock {\em arXiv preprint arXiv:1212.4976}, 2012.

\bibitem{gross2015mirror}
Mark Gross, Paul Hacking, and Sean Keel.
\newblock Mirror symmetry for log calabi-yau surfaces i.
\newblock {\em Publications Math{\'e}matiques de l'IHES}, 122(1):65--168, 2015.

\bibitem{gross2010tropical}
Mark Gross, Rahul Pandharipande, Bernd Siebert, et~al.
\newblock The tropical vertex.
\newblock {\em Duke Mathematical Journal}, 153(2):297--362, 2010.

\bibitem{mandel2015scattering}
Travis Mandel.
\newblock Scattering diagrams, theta functions, and refined tropical curve
  counts.
\newblock {\em arXiv preprint arXiv:1503.06183}, 2015.

\bibitem{mandel2019tropical}
Travis Mandel and Helge Ruddat.
\newblock Tropical quantum field theory, mirror polyvector fields, and
  multiplicities of tropical curves.
\newblock {\em arXiv preprint arXiv:1902.07183}, 2019.

\bibitem{mikhalkin2005enumerative}
Grigory Mikhalkin.
\newblock Enumerative tropical algebraic geometry in {R}$^2$.
\newblock {\em Journal of the American Mathematical Society}, 18(2):313--377,
  2005.

\bibitem{mikhalkin2017quantum}
Grigory Mikhalkin.
\newblock Quantum indices and refined enumeration of real plane curves.
\newblock {\em Acta Mathematica}, 219(1):135--180, 2017.

\bibitem{nishinou2006toric}
Takeo Nishinou and Bernd Siebert.
\newblock Toric degenerations of toric varieties and tropical curves.
\newblock {\em Duke Mathematical Journal}, 135(1):1--51, 2006.

\bibitem{tyomkin2017enumeration}
Ilya Tyomkin.
\newblock Enumeration of rational curves with cross-ratio constraints.
\newblock {\em Advances in Mathematics}, 305:1356--1383, 2017.

\bibitem{welschinger2005invariants}
Jean-Yves Welschinger.
\newblock Invariants of real symplectic 4-manifolds and lower bounds in real
  enumerative geometry.
\newblock {\em Inventiones mathematicae}, 162(1):195--234, 2005.

\end{thebibliography}

\end{document}